\documentstyle[12pt, amssymb,amsthm]{article}

\DeclareSymbolFont{rsfs}{U}{rsfs}{m}{n}
\DeclareSymbolFontAlphabet{\mathscr}{rsfs}

\def\gg{{\mathfrak g}} 
 
\def\ii{{\mathfrak i}}

\def\rr{{\mathfrak r}}

\def\sl{{\mathfrak sl}}

\def\VV{{\mathfrak V}}
\def\WW{{\mathfrak W}}

\def\Vir{{\VV \ii \rr}}

\def\CCC{{\mathbb C}}

\def\NNN{{\mathbb N}}

\def\QQQ{{\mathbb Q}} 
\def\RRR{{\mathbb R}} 
\def\SSS{{\mathbb S}}

\def\ZZZ{{\mathbb Z}} 

\def\A{{\mathcal A}}
\def\B{{\mathcal B}}

\def\F{{\mathcal F}}
\def\G{{\mathcal G}}
\def\H{{\mathcal H}}
\def\I{{\mathcal I}}

\def\M{{\mathcal M}}
\def\N{{\mathcal N}}

\def\R{{\mathcal R}}
\def\S{{\mathcal S}}
\def\T{{\mathcal T}}
\def\U{{\mathcal U}}
\def\V{{\mathcal V}}

\def\X{{\mathcal X}}
\def\Y{{\mathcal Y}}

\def\t1{{\frac{1}{3}}}
\def\1/2{{\frac{1}{2}}}
\def\3/2{{\frac{3}{2}}}

\def\Ž{{\' e}}
\def\{{\` e}}
\def\{{\^ e}}
\def\ˆ{{\` a}}

 \newtheorem{lemma}{Lemma}[section]
 \newtheorem{proposition}[lemma]{Proposition}
  
  \newtheorem{example}[lemma]{Example} 
  \newtheorem{theorem}[lemma]{Theorem}
\newtheorem{definition}[lemma]{Definition}
\newtheorem{notation}[lemma]{Notation}

 \newtheorem{corollary}[lemma]{Corollary}
 \newtheorem{remark}[lemma]{Remark}

   \newtheorem{reminder}[lemma]{Reminder}

\title{Neveu-Schwarz and operators algebras III  \\ \textit{Subfactors and Connes fusion}}

\author{S\'ebastien Palcoux }

 \date{}   
                  
\begin{document}
\maketitle

\begin{abstract}
This paper is the third of a series giving a self-contained way from the Neveu-Schwarz algebra  to a new series of irreducible subfactors.
Here we introduce the local von Neumann algebra of the Neveu-Schwarz algebra,  to obtain Jones-Wassermann subfactors for each representation of the discrete series. Then using primary fields we prove the irreducibility of these subfactors; to next compute the Connes fusion ring and obtain the explicit formula of the subfactors indices.

\end{abstract}

\tableofcontents      \newpage

\section{Introduction}

\subsection{Background of the series}
In the $90$'s,  V. Jones and A. Wassermann started a program whose goal is to understand the unitary conformal field theory from the point of view of operator algebras (see \cite{2c}, \cite{2b}).  In \cite{2}, Wassermann defines and computes the Connes fusion of the irreducible positive energy representations of the loop group $LSU(n)$ at fixed level $\ell$, using primary fields, and with consequences in the theory of subfactors.  In  \cite{vtl} V. Toledano Laredo proves the Connes fusion rules for $LSpin(2n)$ using similar methods. Now, let Diff$(\SSS^{1})$ be the diffeomorphism group on the circle, its Lie algebra is the Witt algebra $\WW$ generated by $d_{n}$ ($n \in \ZZZ$), with $[d_{m} , d_{n}] = (m-n)d_{m+n}$. It admits a unique central extension called the Virasoro algebra $\Vir$. Its unitary positive energy representation theory and the character formulas can be deduced by a so-called Goddard-Kent-Olive (GKO) coset construction from the theory of $LSU(2)$ and the Kac-Weyl formulas (see \cite{1}, \cite{3a}, \cite{1}). In \cite{loke}, T. Loke uses the coset construction to compute the Connes fusion for $\Vir$. 
Now, the Witt algebra admits two supersymmetric extensions $\WW_{0}$ and $\WW_{1/2}$ with central extensions called the Ramond and the Neveu-Schwarz algebras, noted $\Vir_{0}$ and $\Vir_{1/2}$. In this series (\cite{NSOAI}, \cite{NSOAII} and this paper),  we naturally introduce $\Vir_{1/2}$ in the vertex superalgebra context of $L\sl_{2}$, we give a complete proof of the classification of its unitary positive energy representations, we obtain directly their character; then we give the Connes fusion rules, and an irreducible finite depth type II$_{1}$ subfactors for each representation of the discrete series.
Note that we could do the same for the Ramond algebra $\Vir_{0} $, using twisted vertex module over the vertex operator algebra of the Neveu-Schwarz algebra $\Vir_{1/2}$, as  R. W. Verrill  \cite{verrill} and  Wassermann \cite{wass2} do for twisted loop groups.

\subsection{Overview of the paper}
Now, $ \widehat{\gg} $ and $\Vir_{1/2}$ give local superalgebras $ \widehat{\gg}(I) $ and $\Vir_{1/2}(I)$ by smearing with the smooth functions vanishing outside of $I$ a proper interval of $\SSS^{1}$. By Sobolev estimates, the action on the positive energy representations is continuous. We generate their von Neumann algebra, included in an algebra of fermions. By Takesaki devissage and coset construction, we obtain that these algebras are the hyperfinite III$_{1}$ factor, whose the supercommutants are generated by chains of compressed fermions. Also,  there is  Haag-Araki duality on the vacuum, and outside,  a Jones-Wassermann subfactor as a failure of duality.

The compressed fermions are examples of primary fields. We construct them in general from maps intertwining two irreducible representations, dealing with spaces of densities. We see that these maps are completely characterized, bounded and classified by coset for two particular charges $\alpha$, $\beta$. We obtain also their braiding relations, which allow to give the leading term of a kind of OPE for smeared primary fields, which permit,  to have the von Neumann density and the irreducibility of the subfactors.

Then, we obtain irreducible bimodules of local von Neumann algebras, giving the framework to define the Connes fusion. Its rules are a  direct consequence of the transport formula (explaining the intertwining for chains), which is proved by the braiding relations and the von Neumann density. The rules give the dimension of the space of primary fields, they show also that the subfactors are finite index, explicitly given by the square of the quantum dimension, a fusion ring character given as unique positive eigenvalue of a fusion matrix, and  a product of two quantum dimensions of $LSU(2)$ by Perron-Frobenius theorem.

\subsection{Main results}
Let $p = 2i+1$, $q = 2j+1$ and $m =  \ell +2$,  we note $H_{ij}^{\ell} $ the $L^{2}$-completion of $L( c_{m} , h_{pq}^{m})$.
We define the Connes fusion $\boxtimes$ on the discrete series representations of charge $c_{m}$, as bimodules of the hyperfinite III$_{1}$-factor generated by the local Neveu-Schwarz algebra $\Vir_{1/2}(I)$, with $I$ a proper interval of $\SSS^{1}$.
\begin{theorem} (Connes fusion) 
\begin{displaymath}   H_{ij}^{\ell}   \boxtimes   H_{i'j'}^{\ell} = \bigoplus_{ (i'', \hspace{0,1cm}  j'') \in \langle i , i'  \rangle_{\ell}  \times \langle j , j'  \rangle_{\ell + 2}  } H_{i''j''}^{\ell}          \end{displaymath}   
with $\langle a , b \rangle_{n} = \{  c = \vert a-b \vert, \vert a-b \vert + 1, ..., a+b  \ \vert \ a+b+c \le n  \}$.
  \end{theorem}
  
  Let $\M_{ij}^{\ell}(I)$ be the von Neumann algebra generated on $H_{ij}^{\ell}$, by the bounded function of the self-adjoint operators of $\Vir_{1/2}(I)$. 
\begin{theorem}  (Haag-Araki duality on the vacuum) 
\begin{displaymath}   \M_{00}^{\ell}(I) =   \M_{00}^{\ell}(I^{c})^{\natural}       \end{displaymath}  
with $X^{\natural}$ be the supercommutant of $X$.
  \end{theorem}
  As a failure of Haag-Araki duality out of the vacuum, we have:
\begin{theorem} (Jones-Wassermann subfactor) 
\begin{displaymath} \M_{ij}^{\ell}(I) \subset   \M_{ij}^{\ell}(I^{c})^{\natural}       \end{displaymath}    
It's a finite depth, irreducible,  hyperfinite III$_{1}$-subfactor, isomorphic to  the hyperfinite III$_{1}$-factor $\R_{\infty}$ tensor  the II$_{1}$-subfactor : 
\begin{displaymath}  (\bigcup \CCC \otimes End_{\Vir_{1/2}} (H_{ij}^{\ell})^{\boxtimes n})'' \subset (\bigcup End_{\Vir_{1/2}} (H_{ij}^{\ell})^{\boxtimes n+1})''      \end{displaymath}  
 of  index  $\frac{sin^{2}(p \pi /  m)}{sin^{2}(\pi / m)}.  \frac{sin^{2}(q \pi /  (m+2))}{sin^{2}(\pi / (m+2))}$,  with \ $p = 2i+1$, $q = 2j+1$, $m =  \ell +2$.
 \end{theorem}

\subsection{Local von Neumann algebras}  For the loop algebra $L\gg$ and the Virasoro algebra $\Vir$, we can work with the cooresponding groups: $LG$ and Diff$(\SSS^{1})$. For the Neveu-Schwarz algebra, there is no group corresponding to the supergenerators $G_{r}$, and so we need to work with unbounded operators.  From the $\gg$-supersymmetric algebra $\widehat{\gg}$, we build a local Lie superalgebra $\widehat{\gg}(I)$, with $I$ a proper interval of $\SSS^{1}$, by smearing with the smooth functions vanishing outside of $I$.  In the same way, we define the local Neveu-Schwarz Lie superalgebra $\Vir_{1/2}(I)$. Thanks to Sobolev estimates, these local algebras (containing unbounded operators) are represented continuously on the $L_{0}$-smooth completion of their positive energy representation. Now, we define the von Neumann algebras generated by these local algebras  as the von Neumann algebra generated by the bounded functions of our self-adjoint operators;  they are $\ZZZ_{2}$-graded von Neumann algebras. Now, $\widehat{\gg}$ acts on  a complex and real fermionic Fock space which decomposes into all its irreducible positive energy representations (with multiplicity spaces), and by coset construction we can do the same with $\Vir_{1/2}$. Then, we see that the previous von Neumann algebras  are included with conditional expectation in a big von Neumann algebra $\M(I)$  generated  by smeared real and complex fermions, which is known (by \cite{2} and a doubling construction) to be the hyperfinite III$_{1}$ factor; now, the modular action is ergodic, so by Takesaki devissage, $\N(I) = \pi(\Vir_{1/2}(I))''$ is also the hyperfinite III$_{1}$ factor, and by the definition of type III, so is for every subrepresentations, so in particular for $\pi_{i}(\Vir_{1/2}(I))''$, with $\pi_{i}$ a generic irreducible positive energy representation.  We deduce  local equivalence, ie, the discrete series representations are unitary equivalent when they are restricted to $\Vir_{1/2}(I)$; we deduce also Haag-Araki duality: 
\begin{center} $\pi_{0}(\Vir_{1/2}(I^{c}))^{\natural} = \pi_{0}(\Vir_{1/2}(I))''$ \end{center} with $X^{\natural}$ the supercommutant of $X$, from the known Haag-Araki duality of $\M(I)$, because the vacuum vector of $H_{0}$ is invariant by the modular operator $\triangle$ of $\M(I)$.   Outside of the vacuum, we have a Jones-Wassermann subfactor: \begin{center} $ \pi_{i}(\Vir_{1/2}(I))'' \subset \pi_{i}(\Vir_{1/2}(I^{c}))^{\natural}$  \end{center}as a failure of Haag-Araki duality. 
\subsection{Primary fields}  
Let $p_{0}$ be the projection on the vacuum representation $H_{0}$. The Jones relation $p_{0} \M(I) p_{0} = \N(I)p_{0}$, implies that
 $ \pi_{0}(\Vir_{1/2}(I))''$ is generated by products of compressed real and complex fermions: $p_{0} \psi_{1}(f_{1}) p_{i_{1}} \psi_{2}(f_{2}) p_{i_{2}}... \psi_{n}(f_{n})p_{0}$, with $p_{i}$ the projection on $H_{i} \subset H$ and $f_{s}$ localized in $I$. The $p_{i} \psi(f) p_{j}$ are bounded operators intertwining the action of $\Vir_{1/2}(I^{c})$ between the representations $H_{i}$ and $H_{j}$.  We want to interpret these compressions as smeared primary fields.
 We define a primary field as  a  linear operator: 
 \begin{center} $\phi^{k}_{ij} :  H_{j} \otimes \F^{\sigma}_{\lambda , \mu}  \to  H_{i}$  \end{center} that superintertwines the action of $\Vir_{1/2}$; with $H_{i}$, $H_{j}$on the discrete series  of $\Vir_{1/2}$ ($k$ is called the charge of $\phi^{k}_{ij}$), and   $\F^{\sigma}_{\lambda , \mu} $ an ordinary representation of $\Vir_{1/2}$ with base $(v_{i})_{i \in \ZZZ +\frac{\sigma}{2}}$, $(w_{j})_{j \in \ZZZ + \frac{1-\sigma}{2} }$, and: 
  \\   \textbf{(a)}   $L_{n}.v_{i}= -(i+\mu +  \lambda n)v_{i+n}$ 
\\  \textbf{(b)}   $G_{s}.v_{i} = w_{i+s} $
\\  \textbf{(c)}  $L_{n}.w_{j} = -(j +\mu + ( \lambda - \frac{1}{2})n) w_{j+n}$
\\  \textbf{(d)} $G_{s}.w_{j}= -(j+\mu + (2 \lambda - 1)s )v_{j+s}$
\\  with $\lambda = 1-h_{k} $, $\mu = h_{j} - h_{i} $,  $\sigma = 0 , 1$.  

Let the space of densities $\{ f(\theta) e^{i \mu \theta} (d \theta)^{\lambda} Ê\vert  f \in C^{\infty}(\SSS^{1}) \}$  where  a finite covering of Diff$(\SSS^{1})$ acts by reparametrisation $\theta \to \rho^{-1}(\theta)$ (if $\mu \in \QQQ$). Then its Lie algebra acts on too, so that it's a $\Vir$-module vanishing the center. Finally, an equivalent construction with superdensities gives a model for $\F^{\sigma}_{\lambda , \mu} $ as $\Vir_{1/2}$-module.

This primary field is equivalent to general vertex operators  $\phi_{ij}^{k}(z)$ (called the ordinary part) and $\theta_{ij}^{k}(z) = [G_{-1/2} , \phi_{ij}^{k}(z)]$ (called the super part),  and we prove that for $i, j, k$ and $\sigma $  fixed,   such operators are completeley characterized by some compatibility conditions, so the space of primary fields associated is at most one dimensional. Note that $\sigma = 0$ gives $\phi_{ij}^{k}$ integer moded  and $\sigma = 1$, half-integer moded. For charge $\alpha = (1/2 , 1/2)$, we build these operators in the following way (an adaptation of an idea of Loke for $\Vir$ \cite{loke}, simplify by A. Wassermann): we start from the GKO coset construction $\F_{NS}^{\gg}  \otimes H_{i}^{\ell} =  \bigoplus H_{ii'}^{\ell} \otimes  H_{i'}^{\ell +2}$, we take the vertex primary field of $LSU(2)$ of level $\ell$ and spin $1/2$:  $I \otimes \phi_{ij}^{1/2, \ell}(z,v) : \F_{NS}^{\gg}  \otimes H_{j}^{\ell} \to \F_{NS}^{\gg}  \otimes H_{i}^{\ell}$,  with $v \in V_{1/2}$ (the vector representation of $SU(2)$).  Let $p_{i'} $ be the projection on the block $H_{ii'}^{\ell} \otimes  H_{i'}^{\ell +2}$. By compatibility relations and unicity,  $p_{i'} (I \otimes \phi_{ij}^{1/2, \ell }(z,v) ) p_{j'}  =
C.z^{r}\phi_{ii'jj'}^{\alpha}(z) \otimes \phi_{i'j'}^{\1/2  , \ell+2 }(z,v)$, with   $C $ a constant possibly zero and $r\in \QQQ$.  Now, $I \otimes \phi_{ij}^{1/2 ,  \ell }(z,v) = \sum_{i'j'}p_{i'} (I \otimes \phi_{ij}^{1/2 \ell }(z,v) ) p_{j'} $, so at least one is non-zero. More precisely, we prove by an irreducibility argument that $\forall j' $, $\exists i'$ with a non-zero term, and so $\phi_{ii'jj'}^{\alpha}(z)$ non-zero. Note that the simple locality relations between non-compressed smeared fermions concentrated on disjoint intervals ( ie $\psi(f)\psi(g) = -\psi(g)\psi(f)$), admit a bit more complicated equivalent after compression: the braiding relations.

 Now using the same idea as Tsuchiya-Nakanishi \cite{tsunaka}, we deduce the braiding relations for $\Vir_{1/2}$: its braiding matrix is the braiding matrix for $LSU(2)$ at level $\ell$, times the transposed of the inverse of the braiding matrix for $LSU(2)$ at level $\ell + 2$ (it's proved by the contribution of the inverse of a gauge transformation of the Knizhnik-Zamolodchikov equation for the braiding of $LSU(2)$). Then, we obtain non-zero coefficients:
\begin{center}
 $\phi_{ii'jj'}^{\alpha \ell}(z) \phi_{jj'kk'}^{\alpha \ell}( w) = \sum  \mu_{rr'}  \phi_{ii'rr'}^{\alpha \ell}( w) \phi_{rr'kk'}^{\alpha \ell}(z)$  with $\mu_{rr'}  \ne 0 $.
\end{center}
 Now if $\phi_{ii'jj'}^{\alpha} = 0$ and  $\phi_{ij}^{1/2, \ell }$ and $\phi_{i'j'}^{\1/2  , \ell+2 }$ non-zero, then, the braiding relation of $\phi_{ii'jj'}^{\alpha} $ with its adjoint is zero, but produced some non-zero terms $\phi_{ii'kk'}^{\alpha}$ by the previous irreducibility argument, contradiction. Then, we see that $\phi_{ii'jj'}^{\alpha}$ is non-zero iff $\phi_{ij}^{1/2 \ell }$ and $\phi_{i'j'}^{1/2  , \ell+2 }$ are non-zero, ie, $i'  = i \pm 1/2 $ and $j' = j \pm 1/2$ (up to some boundary restrictions).  Now, for charge $\beta = (0,1)$ and the braiding with $\alpha$, we do the same, from the Neveu-Schwarz fermion field $\psi(u,z) \otimes I$ commuting with $I \otimes \phi_{ij}^{\1/2 ,  \ell }(v,w) $. 
 
 Next, by a convolution argument, the braiding  runs also  with two smeared primary fields concentrate on disjoint intervals.  
 We deduce also that  the von Neumann algebras $\pi_{0}(\Vir_{1/2}(I))''$ are generated by chains of  primary fields. This new characterization is essential to prove the so-called von Neumann density: if $I $ is a proper interval of $\SSS^{1} $ and   $I_{1}$,  $I_{2 } $ are the intervals obtained by removing a point of $I$  then,    $\pi_{i}(\Vir_{1/2}^{I_{1}})'' \vee \pi_{i}(\Vir_{1/2}^{I_{2}})''= \pi_{i}(\Vir_{1/2}(I))''   $. By local equivalence, we only need to prove it on the vacuum, on which  the local algebra on $I$ as generated by chains concentred on $I $. By linearity, the $L^{2}$-context, and a kind of OPE,  we can separate into products of chains on $I_{1} $ and $I_{2}$.  Next the von Neumann density implies the irreducibility of the Jones-Wassermann subfactor: 
$\pi_{i}(\Vir_{1/2}(I))^{\natural} \cap \pi_{i}(\Vir_{1/2}(I^{c}))^{\natural} = \CCC$,  which significate that the representations $H_{i}$ are irreducibles $\Vir_{1/2}(I) \oplus \Vir_{1/2}(I^{c})$-modules.

\subsection{Connes fusion and subfactors}  
Then, the discrete series representations are   irreducibles bimodules over the local von Neumann algebra $\M = \pi_{0}(\Vir_{1/2}(I))''$. We define a relative tensor product called Connes fusion $\boxtimes$ using  a $4$-points functions: \\ÊConsider the  $\ZZZ_{2}$-graded  $\M$-$\M$ bimodule $Hom_{-\M}(H_{0} , H_{i} ) \otimes Hom_{\M-}(H_{0} , H_{j} )$, we define a pre-inner product on by: 
\begin{center} $(x_{1} \otimes y_{1} , x_{2} \otimes y_{2}) =(-1)^{(\partial x_{1} + \partial x_{2}) \partial y_{2} } ( x_{2}^{\star}x_{1}y_{2}^{\star}y_{1} \Omega , \Omega ) $  \end{center}
The $L^{2}$-completion is also a  $\ZZZ_{2}$-graded  $\M$-$\M$ bimodule, called the Connes fusion between  $H_{i}$ and $H_{j}$ and noted $H_{i}  \boxtimes H_{j}$. The fusion is associative.

 We obtain a fusion ring for $\oplus $ and $\boxtimes$. The key tool to compute this fusion is the transport formula which gives explicitly how  the  chains on the vacuum representation, transform into chains on any representations through the intertwining relations. Thanks to the braiding relations known at charge $\alpha $, we are able to prove the transport formula:
  \begin{center} $\pi_{j} (\bar{a}_{0 \alpha}.a_{\alpha 0}) = \sum \lambda_{k} \bar{a}_{jk}.a_{kj} \quad   \textrm{with} \  \lambda_{k} > 0. $ \end{center}
 with $a_{kj}$ a charge $\alpha$ (ordinary part, so even) smeared primary field of $\Vir_{1/2}$ between $H_{j}$ and $H_{k}$ concentrated on $I$,    $\bar{a}_{jk} = a_{kj}^{\star}$, 
and   $\pi_{j}: H_{0} \to H_{j}$ the local equivalence. Now,  $a_{\alpha 0} \in Hom_{-\M}(H_{0} , H_{\alpha} )$, so:

\begin{center} $\Vert a_{\alpha 0} \otimes y \Vert^{2} =  (a_{\alpha 0}^{\star}a_{\alpha 0}y^{\star}y \Omega , \Omega ) = (y^{\star} \pi_{j}(a_{\alpha 0}^{\star}a_{\alpha 0})y\Omega , \Omega) = \sum \lambda_{k} \Vert a_{kj}y\Omega \Vert^{2}$.   \end{center}
Then using the fact that $a_{\alpha 0} \M$ is dense in $Hom_{-\M}(H_{0} , H_{\alpha} )$ (by von Neumann density), a polarization and the irreduibility of the bimodules, we obtain a unitary map between $H_{\alpha} \boxtimes H_{j}$ and $\bigoplus_{k \in \langle \alpha , j \rangle} H_{k}$, with $k \in \langle \alpha , j \rangle$ iff $\phi_{jk}^{\alpha}$ is a non-zero primary field. We obtain the fusion rule with $\alpha$: 
  \begin{center} $H_{\alpha} \boxtimes H_{j} = \bigoplus_{k \in \langle \alpha , j \rangle} H_{k}.         $     \end{center}

  Now, idem, with the  braiding relations  between charge $\alpha$ and $\beta$ primary fields, we obtain a partial transport formula and partial fusion rules with $\beta$:
 \begin{center} $H_{\beta} \boxtimes H_{j} \le \bigoplus_{k \in \langle \beta , j \rangle} H_{k}.       $     \end{center} 
  
 But,  the fusion rules with $\alpha$ permit to compute a character of the fusion ring called the quantum dimension (by Perron-Frobenius theorem). An easy way to comptute the quantum dimensions is to see that the fusion ring for the Neveu-Schwarz algebra at charge $c_{m}$ is the tensor product of the fusion rings for the loop algebra at level $\ell$ and $\ell + 2$ (with $m = \ell + 2$), modulo a period two automorphism. Then the quantum dimensions for the Neveu-Schwarz algebra is a product of the two (coset corresponding) quantum dimensions for the loop algebra: 
 \begin{center}  $d(H_{ij}^{\ell}) = d(H_{i}^{\ell}).d(H_{j}^{\ell +2}) = \frac{sin((2i+1) \pi /  (\ell+2))}{sin(\pi / (\ell + 2))}.  \frac{sin((2j+1)\pi /  (\ell+4))}{sin(\pi / (\ell + 4))}$  \end{center}
 
   Then the quantum dimensions show that these partial rules with  $\beta$ are the exact ones. Next, we see that the rules for $\alpha$ and $\beta$ permit to compute all fusion rules.   Finally, the Jones-Wassermann  III$_{1} $-subfactors are isomorphic to II$_{1}$-subfactors tensor the hyperfinite  III$_{1} $-factor, by H. Wenzl \cite{wenzl} and S. Popa \cite{popa}. These last subfactors are irreducibles,  finite depth and   finite index given by the square of the quantum dimensions.

\newpage
\section{Local von Neumann algebras}
 \subsection{Recall on von Neumann algebras}
  Let $H$ be an Hilbert space and $\A $ a unital $\star$-algebra of bounded operators.
   \begin{definition} The commutant $\A '$ of $\A$ is the set of $b \in B (H)$ such that,  $\forall a \in \A$, then $[a,b] := ab - ba =  0$  \end{definition}
  \begin{definition} The weak operator  topology closure $\bar{\A}$ of $\A$ is the set of $a \in B(H)$ such that $\exists a_{n} \in \A$ with $(a_{n}\eta , \xi) \to (a\eta , \xi)$, $\forall \eta , \xi \in H$.    \end{definition}
 \begin{reminder} (Bicommutant theorem) \  Let  $\M$ be  a unital $\star$-algebra, then: 
 \begin{center}  $\M'' = \M$ $\iff$  $ \bar{\M}=\M$  \end{center} \end{reminder}  
 \begin{definition} Such a $\M$ verifying one of these equivalents properties is called a von Neumann algebra.   \end{definition}
 \begin{definition} A factor is a von Neumann algebra $\M$ with $\M \cap \M'= \CCC$.   \end{definition}
 \begin{reminder} (Murray and von Neumann theorem)  The set of all the factors on $H$ is a standard borelian space $X$ and every von Neumann algebra $\M$ decompose into a direct integral of factors:  $\M = \int_{X}^{\oplus} \M_{x} d\mu_{x}$ \end{reminder}
 \begin{reminder} (Murray and von Neumann's classification of factors) \\ Let   $ \M \subset B(H) $ be a factor. We shall consider $H$ as a representation of $\M'$. Thus subrepresentations of $H$ correspond to projections in $\M$.  If $p,q  \in \M $ are projections, then $pH$ and $qH$ are unitarily equivalent as representations of $\M'$ iff there is a partial isometry $ u  \in \M$ between $pH$ and $qH$; thus $u^{*}u = p$ and $uu^{*} = q$. We can immediately distinguish three mutually exclusive cases. \\
I. $ \  \ $ $ H $ has an irreducible subrepresentation.  \newline 
II. $ \ $ $ H $ has no  irreducible subrepresentation, but has a subrepresentation not equivalent to any proper subrepresentation of itself.  \newline
III. $H $ has no  irreducible subrepresentation and every subrepresentation is equivalent to some proper subrepresentation of itself.  \newline 
 We shall call $\M $ a factor of type I , II or III acording to the above cases.
  \end{reminder}
  
  \begin{reminder} The type I and II corresponds to factors admitting non-trivial trace,  with only integer values on the projectors  for the type I ($M_{n}(\CCC)$ or $B(H)$),  and non-integer  values for the type II (factors generated by ICC  groups for example).  On the type III,  the values are only $0$ or $\infty$. \end{reminder}
  
   \begin{reminder} (Tomita-Takesaki theory) 
     We suppose the existence of a vector $\Omega$ (called vacuum vector) such that $\M\Omega$ and $\M' \Omega$ are dense in $H$ (ie  $\Omega$ is cyclic and separating).
       Let $S: H \to H$ the closure of the antilinear map: $\star : x\Omega \to x^{\star}\Omega$. Then, $S$ admits the  polar decomposition $S = J \triangle^{\1/2}$ with $J$ antilinear unitary, and $\triangle^{\1/2}$ positve; so that $J \M J = \M' $,  $\triangle^{it} \M \triangle^{-it}  = \M$ and  
$\sigma_{t}^{\Omega} (x) =\triangle^{it} x \triangle^{-it} $ gives the one parameter modular group action. 
\end{reminder}
\begin{reminder} (Radon-Nikodym theorem)  Let $\Omega'$ be another vacuum vector, then there exists a Radon-Nikodym map $u_{t} \in \U ( \M)$, define such that $u_{t+s} = u_{t} \sigma_{t}^{\Omega'} (u_{s})$ and  $\sigma_{t}^{\Omega'} (x)  = u_{t} \sigma_{t}^{\Omega}(x) u^{\star}_{t}$.   Then, modulo $Int(\M)$, $\sigma_{t}^{\Omega} $ is independant of the choice of $\Omega$, ie, there exist an intrinsic $\delta : \RRR \to  Out (\M) = Aut(\M)/Int(\M)$.
On type I or II  the modular action is internal, and so $\delta$ trivial. It's non-trivial for type III.     \end{reminder}
\begin{definition} We can then define two invariants of $\M$,  $T(\M) = ker (\delta)$ and  $\S (\M) = Sp(\delta) = \bigcap Sp(\triangle_{\Omega}) \setminus \{ 0Ê\}$ called the  Connes spectrum of $\M$. 
  \end{definition} 
\begin{reminder} (see \cite{connes})  Let $\M$ be a type III factor, then $ \S (M ) = \{1\}$, $\lambda^{\ZZZ} $  or $\RRR_{+}^{\star}$, and then, $\M$ is called a 
 III$_{0}$, III$_{\lambda}$ or III$_{1}$ factor (with $0< \lambda < 1$ ). \end{reminder}
  \begin{reminder} Let $\M \ne \CCC$ be a von Neumann algebra on $(H,\Omega)$ then it's a III$_{1}$ factor if and only if the modular action (i.e. the action of $\RRR$ on $\M$ via $\sigma_{t}^{\Omega} $)  is ergodic (i.e. it fixes only the scalar operators). \end{reminder}
  
   \subsection{$\ZZZ_{2}$-graded von Neumann algebras}
    \begin{definition} \label{quandm} A $\ZZZ_{2}$-graded von Neumann algebra $(\M, \tau)$ is a von Neumann algebra $\M$ given with a period two automorphism: $\tau \in Aut(\M)$ and $\tau^{2} = I$. Now $\forall x \in \M$, $x = x_{0} + x_{1}$ with $x_{0} = \1/2(x+\tau(x))$ and $x_{1} =  \1/2(x-\tau(x))$ called the even and the odd part of $x$. Then $\tau(x_{0}) = x_{0}$ and $\tau(x_{1}) = -x_{1}$. Hence $\M  = \M_{0} \oplus \M_{1}$; if $a \in  \M_{\varepsilon_{1}}$ and $b \in  \M_{\varepsilon_{2}}$ then $a.b \in \M_{\varepsilon_{1} + \varepsilon_{2}} $.   \end{definition}
   \begin{definition} A $\ZZZ_{2}$-graded Hilbert space is an Hilbert space given with a period two unitary operator: $u \in \U(H)$ and $u^{2} = I$, 
so that  $H= H_{0} \oplus H_{1}$, with $H_{0}$ and $H_{1}$  the eigenspaces of $u$ for the eigenvalues $1$ and $-1$. Let $p_{0}$ and $p_{1}$ the corresponding projection, then $u = p_{0}-p_{1} = 2p_{0}-1$. \end{definition} 
    \begin{remark} Let $\M$ be a von Neumann algebra on $H$ with $\Omega$ its cyclic, separating vector. Then a period two automorphism $\tau$ of $\M$ gives a period two unitary operator $u$ of $H$ by $u: x\Omega \to \tau(x) \Omega$. Conversely, a period two  unitary operator $u$ of $H$ with $u.\Omega = \Omega$ gives $\tau \in Aut(\M)$ by $\tau(x) = uxu$.    \end{remark} 
  
  \begin{definition} Let $x \in B(H)$, then, $\tau(x) = uxu$ defined a period two automorphism on $B(H)$. Then as for definition \ref{quandm}, $x = x_{0} +x_{1}$. \\ÊWe see that $x_{0} = p_{0}xp_{0} + p_{1}xp_{1}$   and  $x_{1} = p_{1}xp_{0} + p_{0}xp_{1}$. \end{definition} 
 \begin{definition} (Supercommutator) \\  Let $[x,y]_{\tau} =[x_{0},y_{0}]+[x_{0},y_{1}]+[x_{1},y_{0}]+ [x_{1},y_{1}]_{+}$  \end{definition}
  \begin{remark} A projection $p$ is even, then, $\forall x \in B(H)$, $[x,p]_{\tau} = [x,p]$.\\  In particular $[x,I]_{\tau} = [x,I] = 0$.  \end{remark} 
 \begin{definition} The supercommutant $\A^{\natural}$ of $\A$ is the set of $b \in B (H)$ such that,  $\forall a \in \A$, then $[a,b]_{\tau} = 0$.   \end{definition}
    \begin{definition} Let $\kappa = p_{0} + i p_{1}$  the Klein transformation.   \end{definition}
 \begin{remark} $\kappa$ is unitary, $\kappa^{-1} = \kappa^{\star} = p_{0} - i p_{1}$ and $\kappa^{2} = u$.  \end{remark}
  \begin{remark} \label{relll} $u x_{0} u = x_{0}$, $u x_{1} u = -x_{1}$, 
  $\kappa x_{0} \kappa^{\star} = x_{0}$, $\kappa x_{1} \kappa^{\star} = -i u x_{1}$  \end{remark}
      \begin{lemma} \label{naturals} Let $\A$ be a von Neumann algebra $\ZZZ_{2}$-graded for $\tau$, then: 
      \begin{center}  $\A^{\natural} = \kappa \A' \kappa^{\star}.$  \end{center}   \end{lemma}
      \begin{proof} Let $a \in \A$ and $x \in B(H)$  such that $[x,a] = 0$. \\ By the relations of the remarks \ref{relll}:\\Ê
      If $x$ is even,  then $[\kappa x \kappa^{\star} , a]_{\tau} = [ x  , a]  = 0$. \\
      If $a$ is even, then $[\kappa x \kappa^{\star} , a]_{\tau} =  \kappa [ x  ,\kappa^{\star}  a \kappa ] \kappa^{\star} = \kappa [ x  , a ] \kappa^{\star} = 0 $. \\
      Else,  $[\kappa x \kappa^{\star} , a]_{\tau} =  [-i \tau x, a]_{+} = -i (u xa  + a u x ) =  -i u [x , a] = 0 $\\
      Then,  $\kappa \A' \kappa^{\star} \subset  \A^{\natural}$; idem,  $\kappa^{\star}  \A^{\natural}\kappa  \subset  \A'$; the result follows.  \end{proof}
      \begin{corollary} $\A^{\natural} $ is unitary equivalent to $\A'$.    \end{corollary}
    \begin{proof}  $\kappa$ is a unitary operator. \end{proof}  
    \begin{lemma} \label{superco} Let $(\A, \tau)$ be a $\ZZZ_{2}$-graded von Neumann algebra then: 
      \begin{center}  $\A^{\natural \natural} = \A.$  \end{center}\end{lemma}
       \begin{proof} $\A^{\natural \natural} = \kappa (\kappa \A' \kappa^{\star})' \kappa^{\star} = \kappa \kappa (\A'')\kappa^{\star}\kappa^{\star}  $, because a von Neumann algebra is generated by its projections, and a projection is even, so commute with $\kappa$. Then $\A^{\natural \natural} = u \A u  = \tau(\A) =  \A$.   \end{proof}

   \subsection{Global analysis}
The generic discrete series representation $L(c_{m} , h_{pq}^{m})$ is a prehilbert space of finite level vectors, we note $H_{pq}^{m}$ its $L^{2}$-completion.
\begin{definition} Let $s \in \RRR$, we define the Sobolev norms $\Vert . \Vert_{(s)}$ as follows:  
\begin{center}  $\Vert \xi \Vert_{(s)} := \Vert (I+L_{0})^{s}\xi \Vert  \ \ \ \forall \xi \in L(c_{m} , h_{pq}^{m})  $  \end{center}   \end{definition}   
\begin{remark} $((1+L_{0})^{2s} \xi , \xi )  = \Vert \xi \Vert^{2} _{(s)}$    \end{remark}
\begin{proposition} \label{velob} (Sobolev estimate)    $\exists k_{n}, k_{r}  >0$ such that $\forall \xi \in L(c_{m} , h_{pq}^{m})$: 
\begin{enumerate} 
\item[(a)]  $ \Vert L_{n} \xi \Vert _{(s)} \le k_{n}(1+\vert n \vert)^{\vert s \vert +3/2} \Vert \xi \Vert_{(s+1)}$ 
\item[(b)]$ \Vert G_{r} \xi \Vert _{(s)} \le  k_{r}(1+\vert r \vert)^{\vert s \vert +1/2}  \Vert \xi \Vert_{(s+1/2)}$   \end{enumerate}
\end{proposition}
\begin{proof} (a) See Goodman-Wallach \cite{gw} (proposition 2.1 p 307). \\ 
(b)$2L_{0}= G_{r}G_{-r}+G_{-r}G_{r}$. Then, $2(L_{0}\xi , \xi)= (G_{r}\xi,G_{r}\xi)+(G_{-r}\xi,G_{-r}\xi)$. 
So, $\Vert G_{r} \xi \Vert^{2}\le k_{1}\Vert L_{0}^{1/2} \xi \Vert^{2}$ for any $r$.
 Now, it suffices to show the result for an eigenvector of $L_{0}$:   $L_{0}\xi= \mu \xi$. We can take $r \le \mu$ (otherwise $G_{r}\xi= 0$).   \\
$\Vert G_{r}\xi \Vert_{s}^{2}=\Vert  (1+L_{0})^{s}G_{r}\xi \Vert^{2}\le (1+\mu -r)^{2s}\Vert G_{r}\xi \Vert^{2}\le
 (1+\mu -r)^{2s}k_{1}\Vert  L_{0}^{1/2}\xi \Vert^{2} \le  (1+\mu -r)^{2s}k_{1}\mu \Vert \xi \Vert^{2}\le \frac{ (1+\mu -r)^{2s}}{(1+\mu)^{2s}}k_{1} \Vert \xi   \Vert^{2}_{s+1/2} \le (1+\vert r \vert )^{2 \vert s \vert +1}k_{1}\Vert \xi \Vert_{s+1/2}^{2}  $. 
 \end{proof}
 \begin{remark} Thanks to  $L_{n} = [G_{n-1/2},G_{1/2}]_{+} $, we obtain directly the estimate $ \Vert L_{n} \xi \Vert _{(s)} \le k(1+\vert n \vert)^{\vert 2s \vert +1} \Vert \xi \Vert_{(s+1)}$  without Goodman-Wallach result.  \end{remark}
 \begin{definition} Let $H_{pq}^{m,s}$ be the $\Vert .  \Vert_{s}$-completion of $L(c_{m} , h_{pq}^{m})$ and: 
 \begin{displaymath} \H_{pq}^{m}  = \bigcap_{s>0}  H_{pq}^{m,s}   \end{displaymath} with the usual Fr\'echet topology from the norms  $\Vert .  \Vert_{s}$  \end{definition}
 \begin{corollary}  $L(c_{m} , h_{pq}^{m})$ extends to a continuous representation of $\Vir_{1/2}$ on $\H_{pq}^{m} $.   \end{corollary}

\begin{definition}  Let $d = -i\frac{d}{d\theta}$ the unbounded operator of $L^{2}(\SSS^{1})$, let  $F $ be the subspace of finite Fourier series  as a dense domain of $d$.  Let $s \in \RRR$ and $\Vert f \Vert_{(s)}:= \Vert (I + \vert \delta \vert)^{s}.f  \Vert_{1}  $ a Sobolev norm on $F$. Let $F_{s}$ be the completion of $F$ relative to $\Vert . \Vert_{(s)}$. Idem for $e^{i \theta/2}F$.
 \end{definition}
 \begin{definition} Let $L_{f}=  \sum a_{n} L_{n}$ and $G_{h} = \sum b_{r} G_{r}$  such that $f(\theta) = \sum a_{n}e^{in\theta}$, $h(z) =  \sum b_{r}e^{ir\theta} $ and  $f \in F$ and $h \in e^{i \theta/2} F$. \end{definition}
 \begin{notation}  Let $(f,h)_{\RRR} := \frac{1}{2 \pi i}\int_{0}^{2\pi}f(\theta) h(\theta) d\theta  Ê$, with $f.h \in F$  \end{notation}
 \begin{lemma}\label{lobra} (Lie bracket relation)  
  \begin{center}  $ \left\{  \begin{array}{l} 
\lbrack L_{f},L_{h} \rbrack \hspace{0,3cm} =L_{d(f)h - fd(h)} +\frac{C}{12}((d^{3}-d)(f) , h)_{\RRR}  Ê \\
\lbrack G_{f},L_{h} \rbrack \hspace{0,28cm} = G_{d(f)h-\1/2 fd(h))}    \\
\lbrack G_{f},G_{h}\rbrack_{+} = 2L_{fh} +\frac{C}{3}((d^{2}-1)(f) , h)_{\RRR} 
\end{array}   \right.  $ \end{center}
The $\star$-structure:    $L_{f}^{\star}  = L_{ \bar{f} } $,  $G_{h}^{\star} = G_{ \bar{h} }$.    \end{lemma}
 \begin{proof} Direct by computation from proposition 2.9 of \cite{NSOAI}. \end{proof}
 \begin{proposition}  (Sobolev estimate)   \\Ê $\exists k>0$ such that $\forall \xi \in H_{pq}^{m}$ and $f \in F$, $h \in e^{i \theta/2}F$: 
\begin{enumerate} 
\item[(a)]  $ \Vert L_{f} \xi \Vert _{(s)} \le  k \Vert f \Vert_{(\vert s \vert + 3/2)} \Vert \xi \Vert_{(s+1)}$ 
\item[(b)]$ \Vert G_{h} \xi \Vert _{(s)} \le  k  \Vert h \Vert_{(\vert s \vert + 1/2)} \Vert \xi \Vert_{(s+1/2)}$   \end{enumerate}
\end{proposition}
\begin{proof} It's immediate from proposition \ref{velob}.    \end{proof}
\begin{reminder} $\bigcap_{s >0} F_{s} = C^{\infty}(\SSS^{1})$.   \end{reminder}
\begin{corollary}  \label{smoothi}The operators $L_{f}$ and $G_{h}$ act continuously  on $\H_{pq}^{m}$, with $f \in C^{\infty}(\SSS^{1})$ and $h \in e^{i \theta/2}C^{\infty}(\SSS^{1})$.  \end{corollary}

 \begin{reminder} Let $T$ be an operator on a Hilbert space $H$. A subspace $D(T)$ of $H$ is called a domain of $T$ if $T.D(T) \subset H$. Then let $\Gamma (T) = \{ (x,T.x) , x \in D(T) \}$ be the graph of $T$. The operator $T$ is closed if its graph $\Gamma (T)$ is closed in $H \times H $. An operator $\tilde{T}$ is an extension of $T$ if $\Gamma(T) \subset \Gamma(\tilde{T})$, we write  $T \subset \tilde{T}$. The operator $T$ is closable if it admits a closed extension; let $\bar{T}$ be the smallest one. Then, $T$ is closable iff  $\overline{\Gamma(T) }$ is the graph of a linear operator (not always true).  If $T$ is densely defined, then its adjoint $T^{\star}$ is closed because its graph is an orthogonal. From now, every domain is dense in $H$. The operator $T$ is symmetric or formally self-adjoint if $T \subset T^{\star}$, essentially self-adjoint if $\bar{T} = T^{\star}$, and self-adjoint if $T = T^{\star}$. 
  \end{reminder}
 \begin{reminder} \label{gfn} (Glimm-Jaffe-Nelson commutator theorem \cite{resi} X.5) \\Ê
 Let $D$ be a diagonalizable, positve, compact resolving operator and $X$  formally self-adjoint, with common dense domain. If  $(D+I)^{-1}X$, $X(D+I)^{-1}$ and   $(D+I)^{-1/2}[D, X](D+I)^{-1/2}$   are bounded, then $X$ is essentially self-adjoint.\end{reminder}
\begin{lemma}\label{gjncons}Let $f$, $h \in C^{\infty}(\SSS^{1})$ and real, then, $L_{f}$ and $G_{h}$ act on $\H_{pq}^{m}$ as essentially self-adjoint operators.  \end{lemma}
\begin{proof} The function $f$ is real, so $\bar{f}  = f$, then, by the $\star$-structure and the unitarity of the action, $L_{f}$ is formally self-adjoint. Now, $L_{0}$ is positive and 
by Sobolev estimate: $\Vert (L_{0} + I)^{-1}L_{f} \xi \Vert  = \Vert L_{f} \xi \Vert_{(-1)} \le k \Vert  \xi \Vert_{(0)} = k \Vert  \xi \Vert$, so   $(L_{0} + I)^{-1}L_{f} $ is bounded.
Now, $\Vert L_{f} \eta \Vert \le  k\Vert \eta \Vert_{1} = k \Vert (L_{0} + I) \eta \Vert$, so taking $\xi = (L_{0} + I) \eta$, we find $\Vert L_{f} (L_{0} + I)^{-1} \xi \Vert \le k \Vert \xi \Vert$. Finally, $[L_{0} , L_{f}] = L_{h}$ with $h(z) = -zf'(z)$, so combining the two previous tips with $\xi  = (L_{0} + I)^{1/2} \eta$, we find  $(L_{0}+I)^{-1/2}[L_{0}, L_{f}](L_{0}+I)^{-1/2} $ bounded too. We can do the same with $G_{h}$ because $\Vert \xi \Vert_{(s+1/2)} \le \Vert \xi \Vert_{(s+1)}$.  Then, the result follows by reminder \ref{gfn}. \end{proof}
\begin{remark}  This result was already known for Diff$(\SSS^{1})$ and hence the $L_{f}$. On the other hand $G_{f}^{2} = L_{f^{2}} + kId$, so the essential self-adjointness  follows by Nelson's theorem:   \end{remark}
\begin{reminder} (Nelson's theorem \cite{nel}) Let $H$ be an Hilbert space,  $A$ and $B$ be formally self-adjoint operator  acting on a dense subspace $D \subset H$, such that $AB \xi = BA \xi$ $\forall \xi \in D$, and $A^{2} + B^{2}$ essentially self-adjoint, then $A$, $B$ are essentially self-adjoint, and their bounded function commute on $H$.  \\ÊRemark that we have the same result for supercommutation introducing  $\kappa$.   \end{reminder}

\begin{reminder} Let $T$ be a self-adjoint operator with $D(T)$ dense in $H$. There exist a finite measure space $(Y, \mu)$, a unitary operator $U: H \to L^{2}(Y,\mu)$ and a real function $f$, finite up to a null set on $Y$, such that, if $M_{f}$ is the operator of multiplication by $f$, with domain $D(M_{f})$, then $\nu \in D(T) \iff U\nu \in D(M_{f})$, and $\forall g \in D(M_{f})$, $UTU^{\star}g = fg$.  Let $h$ be a borelian function bounded on $\RRR$. The bounded operator $h(T)$ on $H$ is defined  by $h(T) = U^{\star} M_{h(f)}U $.   \end{reminder}

\subsection{Definition of local von Neumann algebras}
\begin{definition} \label{dixmier}(Dixmier) Let $H$ be an Hilbert space. An unbounded self-adjoint operator $T$  is affiliated to a von Neumann algebra $\M$ if it satisfy one of the followings equivalent properties:   \begin{enumerate}
\item[(a)] $\M$ contains all the spectral projection of $T$.
\item[(b)] $\M$ contains every bounded functions of $T$.
\item[(c)]  $\forall u \in \M'$ unitary,  $uD(T) = D(T)$ and $uT\xi = T u\xi$, $\forall  \xi \in D(T)$.  
\end{enumerate}
We note $T \hspace{0.07cm}  \eta \hspace{0.07cm}  \M$.
  \end{definition}
  \begin{remark}By lemma \ref{superco}, if $(\M , \tau)$ is a $\ZZZ_{2}$-graded von Neumann algebra, we can add: 
   \begin{enumerate}
\item[(c')]  $\forall u \in \M^{\natural}$ unitary,  $uD(T) = D(T)$, $uT\xi = (-1)^{\partial T \partial u} T u\xi$, $\forall  \xi \in D(T)$.  
\end{enumerate}
   \end{remark}

\begin{definition} Let $I $ be a proper interval of $\SSS^{1}$. \\ÊWe define $C^{\infty}_{I}(\SSS^{1})$ as the algebra of smooth functions vanishing out of $I$.  \end{definition}
\begin{definition} Let $\Vir_{1/2}(I)$ be the local Neveu-Schwarz Lie superalgebra, generated by $L_{f}$, $G_{f}$ with $f \in C^{\infty}_{I}(\SSS^{1})$, and $C$ central.   \end{definition}
\begin{lemma} (Locality)  \label{supco2} $\Vir_{1/2}(I) $  and   $\Vir_{1/2}(I^{c})$ supercommute.    \end{lemma}
 \begin{proof}By lemma \ref{lobra}, the computation of the brackets involve product of functions in $C^{\infty}_{I}(\SSS^{1})$ and $C^{\infty}_{I^{c}}(\SSS^{1})$, but $C^{\infty}_{I}(\SSS^{1}) . C^{\infty}_{I^{c}}(\SSS^{1})$ = \{0\}.  \end{proof}
 \begin{definition} Let $p_{0}$ be the projection on the space generated by the vectors of integer level, $p_{1} = 1-p_{0}$, $u  = p_{0}  - p_{1}$ and $\tau(x) = uxu$. \end{definition}
 \begin{definition}  Let the von Neumann algebra $\N_{pq}^{m}(I)$ be the minimal von Neumann subalgebra of $B(H_{pq}^{m})$ such that the self-adjoint operators of $\Vir_{1/2}(I) $ (i.e $L_{f}$, $G_{f}$ with  $f \in C^{\infty}_{I}(\SSS^{1})$ real), are affiliated to it.  See definition \ref{dixmier} for equivalent definitions. $(\N_{pq}^{m}(I), \tau)$ is a $\ZZZ_{2}$-graded von Neumann algebra. Ê \end{definition}
 \begin{corollary} (Jones-Wassermann subfactor)   $\N_{pq}^{m}(I) \subset \N_{pq}^{m}(I^{c})^{\natural}$    \end{corollary}
 \begin{proof} $\Vir_{1/2}(I) $  and   $\Vir_{1/2}(I^{c})$ supercommute, then, by  lemma \ref{gjncons}  and Nelson's theorem,  $G_{f}$ and $G_{g}$ supercommute for $f$ and $g$ concentrated on $I$ and $I^{c}$. So is for the von Neumann algebra they generate.    \end{proof}
 \begin{theorem}  (Reeh-Schlieder theorem)  \ Let $v \in H_{pq}^{m}$ be a non-null vector of finite level, then, $\N_{pq}^{m}(I).v$ is dense in $H_{pq}^{m}$ (i.e. $v$ is a cyclic vector). \end{theorem}
\begin{proof} It's a general principle of local algebra, see \cite{2} p 502.   \end{proof}

\subsection{Real and complex fermions}
    \begin{reminder}  (The complex Clifford algebra, see \cite{2}) Let $H$ be a complex Hilbert space, the complex Clifford algebra Cliff$(H)$ is the unital $\star$-algebra generated by a complex linear map $f \mapsto a(f)$ $f \in H$ satisfying: 
    \begin{center} $ [a(f) , a(g) ]_{+} = 0 $  \quad  and \quad  $  [a(f) , a(g)^{\star} ]_{+} = (f,g)$ \end{center}   
  The complex Clifford algebra as a natural irreducible representation $\pi$ on the fermionic Fock space $\F(H) = \Lambda H = \bigoplus_{n = 0}^{\infty}\Lambda^{n} H$ (with $\Lambda^{0} H = \CCC \Omega$ and $\Omega$ the vacuum vector), given by $\pi(a(f))\omega = f \wedge \omega$ bounded.  
 Let $c(f) = a(f) + a(f)^{\star}$ satisfying  $[c(f) , c(g)]_{+} = 2Re(f,g)$ and generating the real Clifford algebra. Warning, $c$ is only $\RRR$-linear. We have the correspondence $a(f)  = \1/2 (c(f) -ic(if))$. Now if $P$ is a projector on $H$, we can define a new irreducible representation $\pi_{P}$ of the complex Clifford algebra by $\pi_{P}(a(f)) = \1/2 (c(f) -ic(\I f))$, where $\I$ is the multiplication by $i$ on $P H$ and by $-i$ on $(I-P)H$, ie,  $\I = iP -i(I-P) = i(2P-I)$. We know that $\pi_{P}$ and $\pi_{Q}$ are unitary equivalent if $P-Q$ is an Hilbert-Schmidt operator. Now, a unitary $u \in U(H)$ is implemented in $\pi_{P}$ if $\pi_{P}(a(u.f)) = U \pi_{P}(a(f)) U^{\star} $ with $U$ unitary, unique up to a phase. But $\pi_{P}(a(u.f))  = \pi_{Q}(a(f))$ with $Q  = u^{\star}Pu$. Then, $u$ is implemented in $\pi_{P}$ if $[P,u]$ is Hilbert-Schmidt.   \end{reminder}
 
 \begin{reminder} More generally, taking a real Hilbert space $H$, we have the real Clifford algebra: $[c(f) , c(g)]_{+} = 2(f,g)$, $f, g \in H$. Then, we define a complex structure $\I$ with $\I^{2} = -Id$. We obtain the complex Hilbert space $H_{\I}$  and then we can define the complex Clifford algebra by: $A(f) =  \1/2 (c(f) -ic(\I f))$,  acting irreducibly on the fermionic Fock space $\F_{\I} = \Lambda H_{\I}$.
 Now, the quantisation condition is: $u \in O(H)$ is implemented in $\F_{\I}$ if $[u,\I]$ is Hilbert-Schmidt. This quantisation due to Segal can be deduce form the condition on the complex case, using the doubling  construction described below. \end{reminder}
 
 \begin{example} (The Neveu-Schwarz real fermions)  \\
 Let the real Hilbert space of anti-periodic functions $H_{NS} = \{ f : \RRR \to \RRR \vert f(\theta + 2\pi) = -f(\theta)  \}$ with basis, $\{cos(r\theta), sin(r\theta) \vert r \in \ZZZ + 1/2 \} $, let the complex structure $\I$ defined by $\I cos (r \theta) = sin(r \theta)$ and $\I sin (r \theta) = -cos(r \theta)$.  Then we obtain the operators $c(f)$ acting irreducibly on the fermionic Fock space we call $\F_{NS}$. Then, we define $\psi_{n} = c( cos(n \theta)) +  i c (sin (n \theta))$. Now, $\psi_{n}^{\star} = \psi_{-n}$ and $[\psi_{m} , \psi_{n} ]_{+} = \delta_{m+n}Id$. The fermionic Fock space can be identified with the irreducible positive energy representation already studied.   \end{example}
 
 \begin{reminder}  (The doubling construction) This is a precise mathematical version of the following physicists slogan: ``a complex fermion is equivalent to two real fermions ''. 
 We start with a real Hilbert space $H$ and we take $H \oplus iH$ (as a real Hilbert space, $iH$ is the same as $H$). Let $v = \xi \oplus i \eta$, we define a real Clifford algebra by  $c(v) = c(\xi ) +  c(i \eta)$, acting irreducibly on $\F(H) \otimes \F(iH)$. Then we define  $a(v) = \1/2 (c(v) - ic(iv)) $ satisfying the complex Clifford relation on the complex Hilbert space $H \oplus iH$. The operator $\I$ on $H$ extends naturally into a unitary operator on $H \oplus iH$. Now, because $\I^{2} = -Id$, it has the form $\I = i(2P-I)$, with $P$ an orthogonal projection. Then the action of the operator $a(v)$ on $\F(H) \otimes \F(iH)$ can be identified with the representation $\pi_{P}$ above, by the unique unitary sending $\Omega \otimes \Omega$ to $\Omega$.
     \end{reminder} 
     \begin{example}We apply to the previous example: in this case, $H_{NS} \oplus iH_{NS} = \{ f : \RRR \to \CCC \vert f(\theta + 2\pi) = -f(\theta)  \}$. But then the multiplication with $e^{i\theta/2}$ gives an identification with $L^{2}(\SSS^{1} , \CCC)$. This construction was already use on \cite{vtl}.   \end{example}

 \begin{reminder} \label{complexfer} (The local  algebra for complex fermions) Let $V$ be a complex finite dimensional complex vector space and  $H = L^{2}(\SSS^{1} , V)$, let $P$ be the projection on the Hardy space $H^{2}(\SSS^{1} , V)$ (the space of function without negative Fourier coefficient). Let $I$ be a proper interval of $\SSS^{1}$ and  $\M(I)$ be the von Neumann algebra generated  by $\pi_{P}($Cliff$( L^{2}(I , V)))$, then: 
 \begin{enumerate}
  \item[(a)]  (Haag-Araki duality)   \quad    $\M(I)^{\natural} = \M(I^{c}) $
\item[(b)] (Covariance) $u_{\varphi^{-1}}: f \mapsto   \sqrt{\varphi' }.f \circ \varphi $ defines a unitary action of  $\varphi \in$Diff$(\SSS^{1})$  on $H$;  this action is implemented in $\pi_{P}$.
\item[(c)]  The modular action on $\M(I)$ is $\sigma_{t}(x) = \pi_{P}(\varphi_{t}) x \pi_{P}(\varphi_{t})^{\star}$, with $\varphi_{t} \in \textrm{Diff}(\SSS^{1})$ the M$\ddot{o}$bius flow fixing the end point of $I$. For example, if $I$ is the upper half-circle, then $\partial I = \{-1 , +1\}$ and  $\varphi_{t}(z) = \frac{ch(t)z+sh(t)}{sh(t)z+ch(t)}$.
\item[(d)]  The modular action is ergodic (ie it fixes only the scalar operators), so that $\M(I)$ is a III$_{1}$ factor (the hyperfinite one).
\end{enumerate}
    \end{reminder}
    
 \begin{remark}\label{realcova} By the doubling construction,  Diff$(\SSS^{1})$ acts  on $H_{NS} $ by:
 \begin{center} $\pi (\varphi)^{-1}.f = \vert \varphi'\vert^{1/2}f \circ \varphi $  \end{center} 
  and the action is quantised. We verify directly that $H_{\CCC} := H_{NS} \oplus iH_{NS} $ admits the orthogonal basis $e_{r}= e^{ir\theta} $ with $r \in \ZZZ + 1/2$, that $\I = (2P-I)i$, with $P$ the Hardy projection (on  the positive modes $r \ge 0$). Now, the Lie algebra of Diff$(\SSS^{1})$ is the Witt algebra. The infinitesimal version of the previous action is
 $d_{n}e_{r} = -(r+n/2)e_{r+n} $: the action of the Witt algebra on the $1/2$-density (see below or \cite{8} p 4). This infinitesimal action of the Witt algebra is implemented on the Fock space $\F_{\CCC} = \F_{NS} \otimes \F_{NS}$ into the Virasoro derivation on the real fermions:  $[L_{n}, \psi_{r}] = -(r+n/2)\psi_{r+n}$ (consistent with section \ref{sec2}). 
  Let  $SU(1,1) $ be the group of  $g = \left( \begin{array}{cc}\alpha & \beta \\ \bar{\beta} & \bar{\alpha}     \end{array}  \right) $ with $\vert \alpha \vert^{2} - \vert \beta \vert^{2} = 1$. By the Mobius transformation: $g(z) = \frac{\alpha z + \beta}{\bar{\beta} z +  \bar{\alpha}  }$, $SU(1,1)$
is injected in   Diff$(\SSS^{1})$, and its Lie algebra is generated by $d_{-1}, d_{0}, d_{1}$. Now, we can see directly that $SU(1,1)$ is  quantised,  because it acts unitarily  and commutes with $P$:
\begin{center} $ \pi (g)^{-1}f(z) = \frac{1}{\vert \bar{\beta}z + \bar{\alpha}\vert}f(g(z)) $ \end{center}  
 Using $\vert \bar{\beta}z + \bar{\alpha}\vert = ( \bar{\beta}z + \bar{\alpha})^{1/2}(\beta \bar{z}+ \alpha)^{1/2} $,  for $k \ge 0$:
\begin{center}$\pi(g)^{-1}z^{k+1/2}= \frac{(\alpha z + \beta)^{k}z^{1/2}}{(\beta z + \alpha)^{k+1}}\in PH_{\CCC} $   \end{center}
Now, the quantised action of $SU(1,1)$ fixes the vacuum vector of the fermionic Fock space, because  $L_{-1}, L_{0}, L_{1}$ vanish on the vacuum vector. \\ 
Note that the Lie algebra of the modular action is generated by $L_{1} - L_{-1}$.
   \end{remark}  
    
      \begin{reminder} (Takesaki devissage \cite{takesaki2}) Let $M \subset B(H) $ be a von Neumann algebra, $\Omega \in H$ cyclic for $M$ and $M'$, $\triangle^{it}$, $J$ the corresponding modular operators ($\triangle^{it} M \triangle^{-it} = M$ and $JMJ = M'$). If $N \subset M$ is a von Neumann subalgebra such that $\triangle^{it} N \triangle^{-it} = N$ (conditional expectation), then:  \begin{enumerate}
\item[(a)]  $\triangle^{it}$ and $J$ restrict to the modular automorphism group $\triangle_{1}^{it}$ and conjugation operator $J_{1}$ of $N$ for $\Omega$ on the closure $H_{1}$ of $N\Omega$.
\item[(b)]   $\triangle_{1}^{it} N \triangle_{1}^{-it} = N$ and $J_{1}NJ_{1} = N'$ on $H_{1}$.
\item[(c)]  If $p$ is the projection onto $H_{1}$, then $pMp = Np$ and \\Ê $N = \{ x \in M \  \vert \ xp = px     \} $  (the Jones relations \cite{23})
\item[(d)] $ H_{1} = H \iff M = N $
\item[(e)] The modular group fixes the center. In fact $\triangle^{it} x \triangle^{-it} = x$ and $JxJ = x^{\star}$ for $x \in Z(M) = M \cap M'$.
\end{enumerate} \end{reminder}
\begin{definition} Let  $\M_{NS}(I)$ be the von Neumann algebra generated by the real Neveu-Schwarz $\psi_{f}$ with $f$ localised on $I$.     \end{definition}
\begin{lemma} (Reeh-Schlieder theorem) Let $v \in \F_{NS}$ be a non-null vector of finite level, then, $\M_{NS}(I).v$ is dense in $\F_{NS}$ (i.e. $v$ is a cyclic vector).   \end{lemma} \begin{proof} It's a general principle of local algebra, see \cite{2}.   \end{proof}
\begin{reminder} A von Neumann algebra $\M$ is hyperfinite iff it is injective, ie  $\M \subset B(H)$ with conditional expectation (see \cite{connes}).  \end{reminder}
\begin{proposition} The local algebra $\M_{NS}(I)$ satisfy Haag-Araki duality, covariance for Diff$(\SSS^{1})$, and the modular action is geometric and ergodic. In particular,   $\M_{NS}(I)$ is the hyperfinite III$_{1}$ factor    \end{proposition}
\begin{proof} The covariance is shown in remark \ref{realcova}. Then, $\M_{NS}(I)$ is stable by the modular action of $\M(I) $.  Now, $\pi_{P}(\M_{NS}(I)) \subset \M(I) \subset B(H_{\CCC})$ with conditional expectation, so $\pi_{P}(\M_{NS}(I))$ is hyperfinite. Next by Takesaki devissage the modular action of $\pi_{P}(\M_{NS}(I))$ is ergodic, so it's the hyperfinite III$_{1}$ factor. Now, by definition of the type III, every subrepresentations are equivalents, but one copy of $\F_{NS}$ is a subrepresentation. So $\M_{NS}(I)$ is the hyperfinite III$_{1}$ factor. Finally, the Haag-Araki duality for $\M_{NS}(I)$ comes from the Haag-Araki duality for $\M(I)$, the Reeh-Schlieder theorem and the Takesaki devissage.  \end{proof}
\subsection{Properties of local algebras deducable by devissage from loop superalgebras}
In this section we will deduce a few partial results on the local von Neumann algebra of Neveu-Schwarz, using devissage from the loop superalgebras, but it's not enough. In  the next section, we will prove more general definitive result by devissage from real and complex fermions (in particular this will imply all the result proved here).

\begin{remark}   $\F_{NS}^{\gg} = \F_{NS}^{\otimes 3}$.  \end{remark}
\begin{lemma} \label{kms} Let $N_{1}$, $N_{2}$ be von Neumann algebra, with modular action $\sigma^{\Omega_{1}}_{t}$ and  $\sigma^{\Omega_{2}}_{t}$, then, the modular action on $N_{1} \overline{\otimes} N_{2}$ is $\sigma^{\Omega_{1}\otimes \Omega_{2}}_{t} = \sigma^{\Omega_{1}}_{t} \otimes  \sigma^{\Omega_{2}}_{t}$ \end{lemma}
\begin{proof} By KMS uniqueness (see \cite{2} p 493).  \end{proof}
\begin{definition}Let  $L(j , \ell)  \otimes \F_{NS}^{\gg} $ be the irreducible representation of the $\gg$-supersymmetric algebra $\widehat{\gg}$. Let the local von Neumann algebra $ \N_{j}^{\ell}(I)$ generated by   $\pi_{j}^{\ell}(g)\otimes \pi_{NS}^{\gg}(g) $ and $1\otimes x$, with $g \in L_{I}G $ and $x \in \M^{\gg}_{NS}(I)$.        \end{definition}
\begin{proposition}  \label{kyrie}$ \N_{j}^{\ell}(I)=\pi_{j}^{\ell}(L_{I}G)\otimes \M_{NS}^{\gg}(I) $.   \end{proposition}
\begin{proof}  
$\pi_{NS}^{\gg}(g)$  supercommutes with $\M_{NS}^{\gg}(I^{c})$, so by the Haag-Araki duality  $ \pi_{NS}^{\gg}(g) \in \M_{NS}^{\gg}(I) $. We deduce that $\N_{j}^{\ell}(I)$ is  generated by $ \pi_{j}^{\ell}(g)\otimes 1 $ and $1 \otimes x$. The result follows.
 \end{proof}

\begin{theorem} \label{recallll}  Combining the work of A. Wassermann \cite{2} on local loop group and the previous work on Neveu-Schwarz fermions, we obtain
\begin{enumerate}
\item[(a)]   (Local equivalence) ÊFor every representations $H_{j}^{\ell}$, there is a unique $\star$-isomorphism $\pi_{j}^{\ell} : \N_{0}^{\ell}(I) \to \N_{j}^{\ell}(I)$ coming from $\pi_{0}^{\ell}(B_{f}^{a}) \mapsto \pi_{j}^{\ell}(B_{f}^{a}) = U.\pi_{0}^{\ell}(B_{f}^{a}).U^{\star}  $ and $\pi_{0}^{\ell}(\psi_{g}^{a}) \mapsto \pi_{j}^{\ell}(\psi_{g}^{a})  =U.\pi_{0}^{\ell}(\psi_{g}^{a}).U^{\star} $, with $U : H_{0}^{\ell} \to H_{j}^{\ell}$ unitary.     
\item[(b)] (Covariance)   $\varphi \in $Diff$(\SSS^{1})$ acts unitarily on $H_{j}^{\ell}$ with  $\pi_{j}^{\ell}(\varphi) B_{f}^{a}  \pi_{j}^{\ell}(\varphi)^{\star} = B^{a}_{f \circ \varphi^{-1}}$  and  $\pi_{j}^{\ell}(\varphi) \psi_{g}^{b}  \pi_{j}^{\ell}(\varphi)^{\star} = \psi^{b}_{\alpha.g \circ \varphi^{-1}}$, with $\alpha = \sqrt{(\varphi^{-1})' }$, a kind of Radon-Nikodym correction (which preserves the group action) to be compatible with the Lie structure, ie be unitary on $L^{2}(\SSS^{1})_{\RRR}$.  
 \item[(c)] The modular  action on $\N_{0}^{\ell}(I)$ is  $\sigma_{t}(x) = \pi_{0}^{\ell}(\varphi_{t}) x  \pi_{0}^{\ell}(\varphi_{t})^{\star}$, with $\varphi_{t} \in \textrm{Diff}(\SSS^{1})$ the M$\ddot{o}$bius flow fixing the end point of $I$. For example, if $I$ is the upper half-circle, then $\partial I = \{-1 , +1\}$ and  $\varphi_{t}(z) = \frac{ch(t)z+sh(t)}{sh(t)z+ch(t)}$.
\item[(d)] $\N_{j}^{\ell}(I)$ is the hyperfinite III$_{1}$ factor.
\item[(e)]  $\N_{0}^{\ell} (I) =  \N_{0}^{\ell} (I^{c})^{\natural}  $   Ê(Haag-Araki duality)
\item[(f)]  $\N_{j}^{\ell} (I) \subset  \N_{j}^{\ell} (I^{c})^{\natural}  $ (Jones-Wassermann subfactor)
\item[(g)]  $\N_{j}^{\ell} (I)^{\natural} \cap  \N_{j}^{\ell} (I^{c})^{\natural}  = \CCC $   (irreducibility of the subfactor)
\end{enumerate}
   \end{theorem}
 
   \begin{lemma} The operators $G_{f}$ and $L_{h}$ act continuously on $\H_{j}^{\ell}$, the $L_{0}$-smooth completion of $L(j , \ell)  \otimes \F_{NS}^{\gg} $.  \end{lemma}
\begin{proof} $\H_{j}^{\ell}$ decompose into some irreducible smooth representations of the discrete series $(\H_{pq}^{m})$, the result follows by corollary \ref{smoothi}    \end{proof}
   \begin{notation}  Let $p = 2j+1$, $q = 2k+1$ and $m = \ell + 2$, then, from now, we can note $\H_{pq}^{m}$ as  $\H_{jk}^{\ell}$. It will be a more conveniant notation for the fusion rules computations \end{notation}
\begin{reminder} \label{ktcoset} (Kac-Todorov coset construction) (\cite{NSOAII} section 2.2 or \cite{8b}). 
 \begin{center}$ \H_{0}^{0} \otimes \H_{j}^{\ell} = \bigoplus_{\stackrel{1  \le q \le m+1}{p \equiv q \lbrack 2 \rbrack} } \H_{jk}^{\ell} \otimes  \H_{k}^{\ell +2}$,  and  \end{center}
\begin{center} $\pi_{0}^{0}(G_{f}) \otimes I + I \otimes \pi_{j}^{\ell}(G_{f}) = \sum [ \pi_{jk}^{\ell}(G_{f}) \otimes I   + I \otimes \pi_{k}^{\ell + 2}(G_{f})]$    \end{center}
  \end{reminder}   
 \begin{lemma}\label{localrelation} We write some usefull relations on $\H_{j}^{\ell}$:\begin{enumerate}  
 \item[(a)]   $[ \psi_{f}^{a}, \psi_{h}^{b} ]_{+} = \delta_{a,b} (f, h)_{\RRR}$
 \item[(b)]  $[B_{f}^{a} , B_{h}^{b} ] =[B^{a} , B^{b}]_{f.h} + (\ell  + 2) \delta_{a,b} (d(f), h)_{\RRR} $
\item[(c)]     $[G_{f}, B_{h}^{a}] = -(\ell + 2)^{1/2} \psi^{a}_{f.d(h)}$
\item[(d)]   $[G_{f}, \psi_{h}^{a}]_{+} = (\ell + 2)^{-1/2} B^{a}_{f.h}$
    \end{enumerate}   \end{lemma}  
       \begin{proof} Direct by computation from section 4 of \cite{NSOAI}.  \end{proof}

 Let $\pi$ be a positive energy representation of the loop superalgebra  $\widehat{\gg}$. We know, it is always of the form $H \otimes \F_{NS}^{\gg}$, where $H$ is a positive energy representation $\sigma$ of $LG$ (non necessarily irreducible). The Clifford algebra acts on the second factor and the loop group acts by tensor product. We have already seen that the von Neumann algebra $\pi(\widehat{\gg}_{I} )''$ is naturally a tensor product of von Neumann algebras (proposition \ref{kyrie}). On the other hand, we have the operators $\pi(L_{f})$,  $\pi(G_{f})$, given by the Sugawara construction (first sections) and  $\pi(\varphi)$ with $\varphi \in $Diff$(\SSS^{1})$. The $L_{f}$ gives a projective representation of the Witt algebra, so exponentiate them give the element of Diff$(\SSS^{1})$. The action of Diff$(\SSS^{1})$ is also given by a tensor product of representation.  
The following property will be fundamental. 
\begin{theorem} $\pi(\varphi) \in \pi(\widehat{\gg}_{I} )'' $ and  $\pi(L_{f})$,  $\pi(G_{f})$ are affilated to $\pi(\widehat{\gg}_{I} )''$, if  $\varphi$ and  $f$ are concentrate on $I^{c}$. \end{theorem}
\begin{remark} We will prove it for Diff$(\SSS^{1})$, and so for the $ L_{f}$, in general,  but for $G_{f}$, only for the vacuum representation (general proof on the next section).      \end{remark}
\begin{proof} For the vacuum representation, it's an immediate consequence of the Haag-Araki duality.
Now, we can restrict to $\pi$ irreducible.   For Diff$(\SSS^{1})$,  because we have Haag-Araki duality on $\F_{NS}^{\gg}$, it's sufficient to prove that $\sigma($Diff$_{I}(\SSS^{1}))''\subset \sigma(L_{I}G )''   $. By local equivalence, there exists a unitary $U $ intertwining $\sigma$ and the vacuum representation $\sigma_{0}$. By Haag duality  and covariance $\sigma_{0}($Diff$_{I}(\SSS^{1}))''\subset \sigma_{0}(L_{I}G )'' $. Then, $U\sigma_{0}(\varphi)U^{\star} \subset \sigma(L_{I}G )'' \subset \sigma(L_{I^{c}}G )'  $. On the other hand, $\sigma(\varphi)\in  \sigma(L_{I^{c}}G )' $. So, $T =\sigma(\varphi^{-1})U\sigma_{0}(\varphi)U^{\star}\in \sigma(L_{I^{c}}G )'  $. But, $T \in \sigma(L_{I}G )' $ by covariance relation. Now, by irreducibility $ \sigma(L_{I}G )' \cap \sigma(L_{I^{c}}G )'= \CCC  $, so $T$ is a constant. The result follows. \end{proof}
\begin{theorem} Haag-Araki duality holds for the Neveu-Schwarz algebra.\end{theorem}
\begin{proof} Let $K_{0}$ be the vacuum representation $\Pi_{0}$ of $\widehat{\gg}\oplus \widehat{\gg}$. The operators $L_{f}$ and $G_{f }$ of the coset construction act on $K_{0}$. We will prove that if $f $ is concentrate on the interval $I$, $G_{f}$ is affiliated with $\Pi_{0}(\widehat{\gg_{I}}\oplus \widehat{\gg_{I}})'' $. Then, because $[G_{f_{1}},G_{f_{2}} ]_{+} = L_{f_{1}f_{2}}+ constant$, $L_{f}$ is also affiliated. By Haag-Araki duality, it suffices to prove that the operators $G_{f}  $ supercommutes with the bosons (elememt of the loop algebra) and the fermions, concentrate on $I^{c}$. Let $A = G_{f}$ and let $B$ be either the bosonic operator or the fermionic operator conjugate by the Klein transformation. They are formally self-adjoint for $f$ real. By relation \ref{localrelation}, they commute formally. By the Sobolev estimates and the Glimm-Jaffe-Nelson theorem, $A^{2} + B^{2}$ is essentially self-adjoint. So Nelson's theorem imply the commutation in term of bounded function. 

Now, by the coset construction, and  the Reeh-Schlieder theorem, the bounded functions of the $G_{f} $ and $L_{f}$ applied on the vacuum vector of $K_{0}$ generate the vacuum positive energy representation of the Neveu-Schwarz algebra. The Haag-Araki duality follows by Takesaki devissage.           \end{proof}

\begin{lemma}(Covariance) \label{requiem}  Let $\varphi \in \textrm{Diff}(\SSS^{1})$, then $\pi_{j}^{\ell}(\varphi) \pi_{j}^{\ell}(G_{f}) \pi_{j}^{\ell}(\varphi)^{\star} =\pi_{j}^{\ell} (G_{\beta.f \circ \varphi^{-1}})$, 
with $\beta = 1/ \alpha$, and $\alpha =  \sqrt{(\varphi^{-1})' } $ and $f \in C^{\infty}(\SSS^{1})$.
  \end{lemma}
\begin{proof}    $\pi_{j}^{\ell}(\varphi) [G_{f} , B^{a}_{h} ] \pi_{j}^{\ell}(\varphi)^{\star} = -(\ell +2)^{-1/2}\psi^{a}_{\alpha. (f \circ \varphi^{-1}). (d(h) \circ \varphi^{-1})} = \\
 -(\ell +2)^{-1/2}\psi^{a}_{\beta. (f \circ \varphi^{-1}). d(h \circ \varphi^{-1})} = [G_{\beta.f \circ \varphi^{-1}} , \pi_{j}^{\ell}(\varphi)B^{a}_{h} \pi_{j}^{\ell}(\varphi)^{\star} ]$ \\Ê
 Idem, $\pi_{j}^{\ell}(\varphi) [G_{f} , \psi^{a}_{h} ]_{+} \pi_{j}^{\ell}(\varphi)^{\star} = [G_{\beta.f \circ \varphi^{-1}} , \pi_{j}^{\ell}(\varphi)\psi^{a}_{h} \pi_{j}^{\ell}(\varphi)^{\star} ]_{+}$.\\ 
Then, by irreducibility,  $\pi_{j}^{\ell}(\varphi) G_{f} \pi_{j}^{\ell}(\varphi)^{\star} -G_{\beta.f \circ \varphi^{-1}}$ is a constant operator; \\Ê it's also an odd operator, so it's zero.     \end{proof} 
\begin{corollary} By the coset construction, the covariance relation runs also on the discrete series representations of the Neveu-Schwarz algebra.  \end{corollary}

\subsection{Local algebras and fermions} \label{richarder}
In \cite{2}, the representation of $LSU(2)$ at level $1$  are constructed using two complex fermions. 
This corresponds to the complex Clifford algebra construction on  $\Lambda( L^{2}(\SSS^{1},\CCC^{2} )) =\F_{\CCC^{2} } $. The level $\ell$ representations are obtained taking  $\F_{\CCC^{2} }^{\otimes \ell}$. Then, the level $\ell$ representations  of the corresponding loop superalgebra are realized on the tensor product of this Fock space and the space $ \F^{\gg}_{NS}$, of three fermions. As vertex superalgebra, the vertex superalgebra of the loop superalgebra defines a vertex sub-superalgebra of the vertex superalgebra of  $\F_{\CCC^{2} } ^{\otimes \ell} \otimes \F^{\gg}_{NS}$.    

Let $H = \{f: \RRR \to \RRR \ \vert f(x+ 2\pi ) = -f(x ) \} $. Let $\F_{NS}^{V} = \Lambda(H \otimes V )$, then, $\F_{NS}^{V_{1}\oplus V_{2}} = \F_{NS}^{V_{1}} \otimes \F_{NS}^{V_{2}}  $. Now, considering the diagonal inclusion $\gg \subset \gg \oplus \gg $,   $H \otimes (\gg \oplus \gg ) \ominus H \otimes \gg =  H \otimes [(\gg \oplus \gg ) \ominus \gg] 
 = H \otimes [(\gg \oplus \gg ) / \gg]   $. Then, we easily seen that in the Kac-Todorov construction  described before: 
 \begin{center} $\F_{NS}^{\gg} \otimes( \F_{NS}^{\gg}  \otimes L(i,\ell))= \bigoplus L(c_{m} , h_{pq}^{m}) \otimes   (\F_{NS}^{\gg}  \otimes L(j,\ell + 2)) $,  \end{center}
we can simplify by a factor $\F_{NS}^ {\gg} $ to obtain the following GKO construction: 
\begin{center}$  \F_{NS}^{\gg}  \otimes L(i,\ell)= \bigoplus L(c_{m} , h_{pq}^{m}) \otimes L(j,\ell + 2)) $,  \end{center}
preserving the coset action of the Neveu-Schwarz algebra. It's also true replacing  $ L(i,\ell) $ by a (non necessarily irreducible) positive energy representation $\H$  of level $\ell $. Then the coset action of the Neveu-Schwarz algebra on $  \F_{NS}^{\gg} \otimes \H $ is described by (see also \cite{3a} p114): 
 \begin{enumerate}  
 \item[(a)]  $  L_{n}^{gko} = L_{n}^{\gg \oplus \gg} - L_{n}^{\gg}$
 \item[(b)]   $G^{gko}(z) = \sum G^{gko}_{r}z^{-r-3/2} = \Phi(\tau_{gko} , z) $\\
 with $ \Phi$ the module-vertex operator on $  \F_{NS}^{\gg} \otimes \H$  (see \cite{NSOAI} section 4.3 ), and  $ \tau_{gko} = (2(\ell + 2 )(\ell + 4) )^{-1/2} (\ell \tau_{1} - 2\tau_{2} )$, with $\tau_{1}$, $\tau_{2}$ as in \cite{NSOAI} definition 4.38 .  \end{enumerate}
 Note that:  $ \Phi(\ell \tau_{1} - 2\tau_{2} , z) =  [\sum_{k}(\ell \psi_{k}(z) \otimes X_{k}(z)  - I \otimes \psi_{k}(z) S_{k}(z) )]     \\ 
 = [\sum_{k}(\ell  \psi_{k}(z) \otimes X_{k}(z) - \frac{i}{3} \sum_{ij}\Gamma_{ij}^{k}I \otimes \psi_{i}(z)\psi_{j}(z)\psi_{k}(z) )]  \\
 = [\ell \sum_{k}( \psi_{k}(z) \otimes X_{k}(z)  -2i \sqrt{2} I \otimes \psi_{1}(z)\psi_{2}(z)\psi_{3}(z)] $.

$\begin{array}{c}   \end{array}$

Now, $\F_{\CCC^{2}}^{\otimes \ell} $ is a level $\ell$ representation of the loop algebra (containing all the irreducibles). We apply the previous GKO construction on $\F_{\CCC^{2}}^{\otimes \ell} \otimes \F_{NS}^{\gg} $. Let $\N(I) = \M(I)^{\otimes \ell }\otimes \M_{NS}^{\gg }(I) $ be the local von Neumann algebra generated by the corresponding real and complex fermions. Let $\pi_ {gko} $ be the coset representation of $\Vir_{1/2}$ on. Now, as previously, $\pi_{gko}(\Vir_{1/2}(I)) $ supercommutes with $\N(I^{c}) $, then by Haag-Araki duality $\pi_{gko}(\Vir_{1/2}(I))'' \subset   \N(I) $. Now, $ \pi_{gko}$ is a direct sum of all the irreducible positive energy representation $\pi_{i}$ (with multiplicities) of the Neveu-Schwarz algebra. As previously (see lemma \ref{requiem}),   $\pi_{gko}(\Vir_{1/2}(I))''$ is stable under the modular action of $\N(I)$. So we can apply the Takesaki devissage. We deduce that $\pi_{gko}(\Vir_{1/2}(I))''$ is the hyperfinite III$_{1}$ factor. By the property of the type III, every subrepresentations of $ \pi_{gko}$  are equivalents; in particular all the $\pi_{i}(\Vir_{1/2}(I))''$ are the hyperfinite III$_{1}$-factor, and are  equivalents to  $\pi_{0}(\Vir_{1/2}(I))''$: it's the local equivalence for the Neveu-Schwarz algebra. Finally, let $\Omega$ be the vacuum vector of $\F_{\CCC^{2}}^{\otimes \ell} \otimes \F_{NS}^{\gg} $, then clearly   $\pi_{gko}(\Vir_{1/2}(I))''\Omega $ is dense (Reeh-Schlieder theorem) on the vacuum representation of $\Vir_{1/2} $ tensor its corresponding multiplicity $M_{0}$. Let $P$ be the projection on, then $P$ commutes with the modular operators (because the vacuum vector is invariant) and with the Klein operator $\kappa$. But by Takeaki devissage $P N(I) P  =  \pi_{gko}(\Vir_{1/2}(I))''  P = [\pi^{\otimes M_{0}} _{0}(\Vir_{1/2}(I))'']$. \\So \  $\kappa J P N(I) P J \kappa^{\star} =P \kappa J  N(I)  J \kappa^{\star} P =P N(I)^{\natural} P =P N(I^{c}) P\\Ê =  [\pi^{\otimes M_{0}} _{0}(\Vir_{1/2}(I^{c}))'']=[\pi^{\otimes M_{0}}_{0}(\Vir_{1/2}(I))^{\natural}]$. The Haag-Araki duality for the Neveu-Schwarz algebra follows.
      
\begin{corollary} (Generalized Haag-Araki duality) \\Ê$\pi_{gko}(\Vir_{1/2}(I))'' = \pi_{gko}(\Vir_{1/2})'' \cap  \N(I)$ \end{corollary}
\begin{corollary}  $\pi_{0}(\Vir_{1/2}(I))''$ is generated by chains of compressed fermions concentrate in $I$.    \end{corollary}
\begin{proof} Immediate from Jones relation: $p_{0}\N(I)p_{0} = \pi_{gko}(\Vir_{1/2}(I))'' p_{0}$.   \end{proof}
Now because $\pi_{gko}$ contains all the irreducible positive energy representations $\pi_{i} $ of charge $c_{m}$ , we deduce that: 
\begin{corollary} Let $\pi$ be the direct sum of all the $\pi_{i}$. \\ÊTo simplify we  note $\pi = \pi_{0}\oplus ...\oplus \pi_{n}$. Then $\A :=  \pi(\Vir_{1/2}(I))'' \\Ê= \{ T = \left( \begin{array}{ccc}    T_{1} & &  0 \\  &\ddots &   \\ 0 & & T_{n}\end{array} \right)    \  \vert $ T supercommutes with $ \B  \} $ with $ \B = \{  \left( \begin{array}{ccc}    S_{11} & \ldots & S_{1n} \\ \vdots & & \vdots \\ S_{n1} & ... & S_{nn}\end{array} \right) $ such that $ S_{ij} $  is a chain of compressed fermions $p_{i}\phi(f) p_{j}$ concentrate on $ I^{c}  \} $       \end{corollary} 
By definition $\A^{\natural}  = \B$. Now, let $q_{i}$ the projection on $\pi_{i}$, then $q_{i} \in \A^{\natural} $, so, $(q_{i} \A) ^{\natural}= p_{i}  \B p_{i} $. Then $(q_{i} \A) ^{\natural} =  \pi_{i}(\Vir_{1/2}(I) )^{\natural} = \{S_{ii} \ \vert ...    \} $. 
\begin{corollary} \label{gloria} $ \pi_{i}(\Vir_{1/2}(I^{c}) )^{\natural}$ is generated by chains of compressed fermions concentrated in $I$.      \end{corollary}

 \begin{remark} In the next section, we will see by unicity that the compression of complex fermions give a primary fields of charge $ \alpha = (1/2 , 1/2)$, and the compression of a real fermions give primary fields of charge $\beta = (0 , 1) $.   \end{remark}   
 \begin{remark} We will see that the supercommutation relation on the vacuum (Haag-Araki duality) is replaced by braiding relations of primary fields. As consequence, we directly see that $\pi_{i}(\Vir_{1/2}(I) )^{\natural}$ and $\pi_{i}(\Vir_{1/2}(I^{c} ) )^{\natural}$ do not necessarily supercommute if $i\ne 0  $. Then, the formulation of the local von Neumann algebra, generated by chains of primary fields (with braiding), shows explicitly the failure of Haag-Araki duality outside of the vacuum.    \end{remark}

 \newpage
\section{Primary fields} 
\subsection{Primary fields for $LSU(2)$} 
This section is an overview of   the primary field theory of $LSU(2)$, for a more detailed exposition  see \cite{2} and \cite{tsukani}. It would be convenient to also cite \cite{pasa} and \cite{frre}.
Let $V$ be a representation of $G=SU(2)$ or $\gg = \sl_{2}$.
\begin{definition} Let $\lambda$, $\mu \in \CCC$, we define the ordinary representations of $L\gg  \rtimes \Vir $ as  $\V_{\lambda , \mu} $, generated by $(v_{i})$,  $v \in V$ and $i \in \ZZZ$, and: 

 \begin{description}
 \item[(a)]  $L_{n}.v_{i}= -(i+\mu + n \lambda)v_{i+n}$ 
 \item[(b)]  $X_{m}.v_{i} = (X.v)_{m+i} $ \quad  ($X \in \gg$)
  \end{description}  \end{definition}
  
   \begin{definition} Let  $L_{i}^{\ell}$ and $L_{j}^{\ell}$ be irreducible representation of $L\gg$, of level $\ell$ and spin $i$ and $j$.  We define a primary field as  a linear operator: 
 \begin{center} $\phi :  L_{j}^{\ell} \otimes \V_{\lambda , \mu}   \to L_{i}^{\ell}$  \end{center}  that intertwines the action of $L\gg  \rtimes \Vir$. We call  $V$ the charge of $\phi$.      \end{definition}
 \begin{reminder}Let $h_{i}^{\ell} = \frac{i^{2} + i}{\ell + 2}$ the lowest eigenvalue of $L_{0}$ on $L_{i}^{\ell}$ (see theorem \ref{fari}). The eigenspace is the $\sl_{2}$-module $V_{i}$.   \end{reminder}
 
 \begin{definition} For $w \in  \V_{\lambda , \mu} $, let $\phi(w) : L_{j}^{\ell} \to L_{i}^{\ell}$  \end{definition}
\begin{lemma} Let $X \in L\gg  \rtimes \Vir $, then $[X, \phi(w)] = \phi(X.v)$    \end{lemma}
\begin{proof}  As for the proof of lemma \ref{asfor}. \end{proof}

\begin{lemma} $\phi $ non-null implies that $ \mu  = h_{j}^{\ell} - h_{i}^{\ell}$. 
 \end{lemma}
 \begin{lemma} (Gradation)
$\phi(v_{n}). (L_{j}^{\ell})_{s+h_{j}^{\ell}} \subset  (L_{i}^{\ell})_{s-n+h_{i}^{\ell}}$ 
 \end{lemma}
 \begin{definition}Let $h = 1-\lambda$ be the conformal dimension of $\phi$, \\Êand  $\triangle = 1- \lambda + \mu = h+ h_{j}^{\ell} - h_{i}^{\ell} $ ; we define:  \begin{center}  $\phi(v,z)=  \sum_{n \in \ZZZ } \phi(v_{n})z^{-n-\triangle} $ \quad $(v \in V)$.  \end{center}
  \end{definition}

 \begin{lemma}(Compatibility condition) 
 \begin{description}
 \item[(a)]  $[L_{n} , \phi(v,z)] = z^{n}[z\frac{d}{dz} + (n+1)h] \phi(v,z)$ 
 \item[(b)]  $[X_{m} , \phi(v,z)] = z^{m} \phi(X.v,z)$   
  \end{description}    \end{lemma} 
\begin{proof}Direct from the definition.       \end{proof}
\begin{lemma}   If   $\widetilde{\phi}(z,v)$ satisfy the  compatibility condition, then,  it gives a primary fields for $LSU(2)$.   \end{lemma}
\begin{proof}It's an easy verification.       \end{proof}
\begin{proposition} (Initial term) A primary field $\phi :  L_{j}^{\ell} \otimes \V_{\lambda , \mu}   \to L_{i}^{\ell}$ with every parameters fixed, is completly determined by  its initial term: 
\begin{center} $\phi_{0} :  V_{j} \otimes V   \to V_{i}$   \end{center}    \end{proposition}
\begin{proof} Idem, by intertwining relation; see \cite{2} p 513 for details.  \end{proof}
\begin{proposition} (Unicity)  If $V=V_{k}$ is irreducible, the space of such primary field is at most one-dimensional.  \end{proposition}
\begin{proof}  $\phi_{0}$ is an intertwining operator, ie, $\phi_{0} \in Hom_{\gg}( V_{j} \otimes V_{k} ,  V_{i} )$ the multiplicity space at $V_{i}$ of  $V_{j} \otimes V_{k} = V_{\vert j-k \vert} \oplus V_{\vert j-k \vert +1} \oplus ... \oplus  V_{ j+k }$ (Clebsch-Gordan),  so at most one-dimensional.  \end{proof}

\begin{remark} As for $\Vir_{1/2}$ (see remark \ref{investiggg}), with $(A_{n}B)(z)$ formula, we define inductively the $L\gg$-module $L_{k}^{\ell}$ from $\phi$.  \end{remark}

\begin{corollary} $\mu  = h_{j}^{\ell} - h_{i}^{\ell}$ and  $1-\lambda = h=h_{k}^{\ell}$.    \end{corollary} 
\begin{definition} We note $\phi$ as $\phi_{ij}^{k \ell}  $,   $\triangle$ as $  \triangle_{ij}^{k \ell}= h_{j}^{\ell} - h_{i}^{\ell}+ h_{k}^{\ell}$.   \end{definition}
We call $\phi$  a primary field of spin $k$; in our work, we just need to consider primary fields of spin $1/2$ and   $1$:    
\begin{proposition} Up to a multiplication by a rational power of $z$:  \\
 (a)   The compression of  complex fermions gives  primary fields of spin $1/2$. \\
 (b)  The compression of  real fermions gives  primary fields of spin $1$.     \end{proposition}
\begin{proof} We just check the compatibility condition. The calculation can also be made on the vertex algebra of the fermions. See also \cite{2} p 515. \end{proof}

\begin{definition} We note $\phi_{ij}^{k \ell}$ be the primary field from $L_{j}^{\ell}$ to $L_{i}^{\ell}$, of spin $k$. It's defined up to a multiplicative constant and is possibly null.     \end{definition}
\begin{reminder} (Constructible primary fields of spin $1/2$ or $1$, see \cite{2}).    
\begin{enumerate}
\item[(a)]  $\phi_{ij}^{\1/2 \ell}$ is non-null iff $j = i \pm 1/2$ and $i+j+1/2 \le \ell $
\item[(b)] $\phi_{ij}^{1 \ell}$ is non-null iff $j = i- 1, i, $ or $ i+1$ and $i+j+1 \le \ell $
\end{enumerate}
with the restriction that:   $0 \le \ i,j \  \le \ell /2$
\end{reminder}
\begin{proposition}Every primary fields $\phi_{ij}^{k \ell}(w) : L_{j}^{\ell} \to L_{i}^{\ell}$ of spin $k=1/2$ or $1$,  are constructibles as  compressions   of complex and real fermions.    respectively. \end{proposition}
\begin{proof} For spin $1/2$ primary fields see \cite{2} p 515.

Now, for spin $1$: note that at level $\ell = 2$, there are only $0$, $1/2$ and $1$ as possible spins.  But, the real Neveu-Schwarz fermions $\F_{NS}^{\gg}$ equals to $L_{0}^{2} \oplus L_{1}^{2} $, and the real Ramond fermions $\F_{R}^{\gg}$ equals to $L_{1/2}^{2}$, as $LSU(2)$ module (see \cite{NSOAII} corollary 5.7 and \cite{3a} p116). Then,  compressions of the fermion field $\psi(z,v)$, with $v \in V_{1} = \gg $  on $\F_{NS}^{\gg}$ or $\F_{R}^{\gg}$ give the spin $1$ primary fields at level $2$, by unicity and compatibility condition.

Now, $L_{j}^{\ell} \otimes L_{k}^{\ell'} =  L_{\vert j-k \vert}^{\ell + \ell'} \oplus L_{\vert j-k \vert +1}^{\ell + \ell'} \oplus ... \oplus  L_{ j+k }^{\ell + \ell'}$, so: 
\begin{enumerate}  
\item[(a)] $\phi_{i, i-1}^{1 \ell + 2} (v)$   is the compression of $\phi_{01}^{1, 2}(v) \otimes I : L_{1}^{2} \otimes  L_{i-1}^{\ell} \to L_{0}^{2} \otimes  L_{i-1}^{\ell} $.
 \item[(b)]   $\phi_{i, i+1}^{1 \ell+ 2} (v)$ is the compression of $\phi_{10}^{1, 2}(v) \otimes I : L_{0}^{2} \otimes  L_{i}^{\ell} \to L_{1}^{2} \otimes  L_{i}^{\ell} $.
\item[(c)]   $\phi_{i, i}^{1 \ell+ 2} (v)$ is the compression of $\phi_{01}^{1, 2}(v) \otimes I : L_{0}^{2} \otimes  L_{i}^{\ell} \to L_{1}^{2} \otimes  L_{i}^{\ell} $.
\end{enumerate}
The result follows. \end{proof}
\begin{corollary}  The primary fields of spin  $k= 1/2$ or $1$ are bounded and identifying the $L^{2}$-completion of  $\V_{\lambda , \mu}$ with $L^{2}(\SSS^{1},V_{k})$, we obtain $\phi(f)$ for $f \in L^{2}(\SSS^{1},V_{k})$, with:  $\Vert \phi(f) \Vert \le K \Vert f \Vert_{2}$. \end{corollary}

\begin{reminder} (Braiding relations) \\  \label{braid1}
In \cite{tsukani} and \cite{2}, the braiding relations of spin $1/2$ primary fields are given by reduced $4$-point functions $f : \CCC \to W$, with $W$ finite dimensional. We give an overview of this theory: 

 Let $F_{j}(z,w)  = (\phi_{kj}^{\1/2 \ell}(u,z).\phi_{ji}^{\1/2 \ell}(v,w)\Omega_{i} , \Omega_{k})$, then by gradation it  equals: 
\begin{center} $\sum_{m \ge 0}(\phi_{kj}^{\1/2 \ell}(u,m).\phi_{ji}^{\1/2 \ell}(v,-m)\Omega_{i} , \Omega_{k})z^{-m-\triangle}w^{m-\triangle'} =
 f_{j}(\zeta)z^{-\triangle}w^{-\triangle'} $ \end{center} with   $ f_{j}(\zeta) =   \sum_{m \ge 0}(\phi_{kj}^{\1/2 \ell}(u,m).\phi_{ji}^{\1/2 \ell}(v,-m)\Omega_{i} , \Omega_{k}) \zeta^{m} $ and $\zeta = w/z$. The function $ f_{j}$ are holomorphic for $\vert \zeta \vert < 1$.  Now, $\phi_{kj}^{\1/2 \ell}$ and $\phi_{ji}^{\1/2 \ell}$ are non-zero field, ie $Hom_{\gg}( V_{j} \otimes V_{1/2} ,  V_{k} )$ and $Hom_{\gg}( V_{i} \otimes V_{1/2} ,  V_{j} )$ are $1$-dimensional space. Then, the set of possible such $j$  generate the space $W = Hom_{\gg}(V_{1/2} \otimes V_{1/2} \otimes V_{i}   ,  V_{j} )$. Then, we consider the vector $ f_{j}(\zeta)$ as a vector in $W$. Let $\tilde{f}_{j} = \zeta^{\lambda_{j}}f_{j}$, (with $\lambda_{j} = (j^{2}+j-i^{2}-i-3/4)/(\ell + 2)$), called the reduced four points functions. $\tilde{f}_{j}(z)$ is defined on $\{z : \vert z \vert < 1, z \not \in [0,1[ \}$. It satisfy the Knizhnik-Zamolodchikov ordinary differential equation, equivalent to the hypergeometric equation of Gauss: \begin{center} $\tilde{f}'(z) = A(z) \tilde{f}(z)$, with $A(z) = \frac{P}{z} + \frac{Q}{1-z}$   \end{center}
with $P$, $Q \in End(W)$. It's proved in \cite{2} section 19, the existence of a holomorphic gauge transformation $g: \CCC \backslash [1, \infty [  \to GL(W)$ with $g(0)  =I $ such that: $g^{-1}Ag - g^{-1}g' = P/z$. The solution of the ODE is then $\tilde{f}(z) = g(z)z^{P}T$, with $T$ an eigenvector of $P$. So, up to a power of $z$, the solutions $\tilde{f}_{j}(z)$ are just the columns of $g(z)$ (in the spectral base of $P$). Now, let $r_{j}(z) = \tilde{f}_{j}(z^{-1})$ on $\{z : \vert z \vert > 1, z \not \in [1,\infty[ \}$, then $r_{j}$ satisfy clearly the equation: 
\begin{center} $r'(z) = B(z) r(z)$, with $B(z) = \frac{Q-P}{z} + \frac{Q}{1-z}$   \end{center} 
The function $r_{j}$ and $ \tilde{f}_{j}$ extend to holomorphic functions on $\CCC \backslash [0 , \infty[$.
It's proved in \cite{2}, that the solutions of these two equations are related by a transport matrix $c= (c_{ij})$ with $c_{ij} \ne 0$,  so that:
\begin{center}$ \tilde{f}_{j}(z) = \sum c_{jm} \tilde{f}_{m}(z^{-1})  $  \end{center} 
We then obtain, up to an analytic continuation, the following equality: 
\begin{center} $ (\phi_{kj}^{\1/2 \ell}(u,z).\phi_{ji}^{\1/2 \ell}(v,w)\Omega_{i} , \Omega_{k}) = \sum c_{jm} (\phi_{km}^{\1/2 \ell}(v,w).\phi_{mi}^{\1/2 \ell}(u,z)\Omega_{i} , \Omega_{k})    $ \end{center}
This relation extends to any finite energy vectors: 
\begin{center} $ (\phi_{kj}^{\1/2 \ell}(u,z).\phi_{ji}^{\1/2 \ell}(v,w)\eta , \xi) = \sum c_{jm} (\phi_{km}^{\1/2 \ell}(v,w).\phi_{mi}^{\1/2 \ell}(u,z)\eta , \xi)    $ \end{center}
This analysis runs idem for braiding relations between spin $1/2$ and spin $1$ primary fields, then:
\begin{theorem} (Braiding relations) \\
Let $(k_{1} , k_{2}) = (1/2 , 1/2)$,  $(1, 1/2) $ or  $(1/2 , 1)$; $v_{1} \in V_{k_{1}}$ and  $v_{2} \in V_{k_{2}}$.
\begin{center}
 $\phi_{ij}^{k_{1} \ell}(v_{1},z) \phi_{jk}^{k_{2} \ell}(v_{2} , w) = \sum  \mu_{r}  \phi_{ir}^{k_{2} \ell}(v_{2} , w) \phi_{rk}^{k_{1} \ell}(v_{1},z)$  with $\mu_{r}  \ne 0 $
\end{center}
To simplify, we don't write the dependence of $\mu_{r}$ on the other coefficients.  \end{theorem}  \end{reminder}
\begin{remark} The way to write the braiding relations is a simplification. In fact, the left side is defined for $\vert z \vert < \vert w \vert$, and the right side for $\vert z \vert > \vert w \vert$, but each sides admit the same rational extension out of $z=w$. The braiding relations generalise the locality of vertex operator (see \cite{NSOAI} definition 3.19).      \end{remark}
\begin{remark}  \label{dotfat}
To prove that all the coefficients are non-null for $(k_{1} , k_{2})  = (1,1)$, we should  solve Dotsenko-Fateev equations (see \cite{vtl}). \end{remark} 
\begin{reminder} (Localised braiding relation) Let $f  \in L^{2}(I,V_{k_{1}})$ and $g \in L^{2}(J,V_{k_{2}}) $, with $I$, $J$ be two disjoint proper intervals of $\SSS^{1}$. 
Using an argument of convolution (as \cite{2} p 516), we can write the following localised braiding relations: 
\begin{center}
 $\phi_{ij}^{k_{1} \ell}(f) \phi_{jk}^{k_{2} \ell}(g) = \sum  \mu_{r}  \phi_{ir}^{k_{2} \ell}(e_{\alpha}g) \phi_{rk}^{k_{1} \ell}(e_{-\alpha}f)$  with $\mu_{r}  \ne 0 $
\end{center}
with $e_{\alpha} = e^{i \alpha \theta} $, $\alpha = h_{i}^{\ell}+  h_{k}^{\ell} - h_{j}^{\ell} -  h_{r}^{\ell}$ and $(k_{1} , k_{2}) $ as previously.
\end{reminder}

\begin{reminder}  (Contragrediant braiding) Let the previous ODE:   \label{braid2}
\begin{center} $\tilde{f}'(z) = A(z) \tilde{f}(z)$, with $A(z) = \frac{P}{z} + \frac{Q}{1-z}$   \end{center}
and the previous gauge relation: $g^{-1}Ag - g^{-1}g' = P/z$. \\ÊIn the same way, we can choose $h(z)$ with $h(0)  = I $ and $hAh^{-1} - h'h^{-1} = -P/z$. This corresponds to  take $-A(z)^{t}$ instead of $A(z)$. But then $(hg)' = [P,hg] / z $, which admits only the constant solutions, but $h(0)g(0)  =I$, so $h(z) = g(z)^{-1}$. Then, the columns of $(g(z)^{-1})^{t}$ are the fundamental solutions of $k'(z) = -A(z)^{t}k(z)$. The transport matrix of this equation is just the transposed of the inverse of the original one, ie $(c^{-1})^{t}$. 
\end{reminder}

\subsection{Primary fields for $\Vir_{1/2}$}
\begin{definition} Let $\lambda$, $\mu \in \CCC$,  $\sigma = 0 , 1$, we define the ordinary representations of $\Vir_{1/2}$ as  $\F^{\sigma}_{\lambda , \mu} $, with base $(v_{i})_{i \in \ZZZ +\frac{\sigma}{2}}$, $(w_{j})_{j \in \ZZZ + \frac{1-\sigma}{2} }$, and: 
 \begin{description}
 \item[(a)]  $L_{n}.v_{i}= -(i+\mu +  \lambda n)v_{i+n}$ 
 \item[(b)]  $G_{s}.v_{i} = w_{i+s} $
 \item[(c)]  $L_{n}.w_{j} = -(j +\mu + ( \lambda - \frac{1}{2})n) w_{j+n}$
 \item[(d)]  $G_{s}.w_{j}= -(j+\mu + (2 \lambda - 1)s )v_{j+s}$
  \end{description}  \end{definition}
  \begin{remark}  Let the space of densities $\{ f(\theta) e^{i \mu \theta} (d \theta)^{\lambda} Ê\vert  f \in C^{\infty}(\SSS^{1}) \}$  where  a finite covering of Diff$(\SSS^{1})$ acts by reparametrisation $\theta \to \rho^{-1}(\theta)$ (if $\mu \in \QQQ$). Then its Lie algebra acts on too, so that it's a $\Vir$-module vanishing the center (see \cite{8}). Finally, an equivalent construction with superdensities gives a model for $\F^{\sigma}_{\lambda , \mu} $ as $\Vir_{1/2}$-module  (see \cite{ik} ).  \end{remark}
   
  \begin{definition} Let  $L_{pq}^{m}$ and $L_{p'q'}^{m}$ on the unitary discrete series of $\Vir_{1/2}$. We define a primary field as  a  linear operator: 
 \begin{center} $\phi :  L_{p'q'}^{m} \otimes \F^{\sigma}_{\lambda , \mu}  \to  L_{pq}^{m}$  \end{center} that superintertwines the action of $\Vir_{1/2}$.     \end{definition}
 \begin{definition} For $v \in \F^{\sigma}_{\lambda , \mu} $, let $\phi(v) : L_{p'q'}^{m}   \to  L_{pq}^{m}$  \end{definition}
\begin{lemma} \label{asfor}Let $X \in \Vir_{1/2}$, then $[X, \phi(v)]_{\tau} = \phi(X.v)$    \end{lemma}
\begin{proof}   We can suppose $X$ to be homogeneous for the $\ZZZ_{2}$-gradation $\tau$. 
Now, $\phi$ superintertwines the action of $\Vir_{1/2}$: $\phi.[X \otimes I + I \otimes X] = (-1)^{\partial X}X.\phi $\\
Let $\xi \otimes v \in  L_{p'q'}^{m} \otimes \F^{\sigma}_{\lambda , \mu} $, then  $\phi.[X \otimes I + I \otimes X] (\xi \otimes v) = [\phi(v)X+ \phi(Xv)]\xi  $ and 
$X.\phi (\xi \otimes v) = X \phi(v)\xi$, then $[X, \phi(v)]_{\tau} = \phi(X.v)$. 
 \end{proof}
 \begin{lemma} $\phi $ non-null implies that $ \mu  = h_{p'q'}^{m} - h_{pq}^{m}$. 
 \end{lemma}
 \begin{lemma} \label{grzgra} (Gradation)
 \begin{enumerate} 
\item[(a)]  $\phi(v_{n}). (L_{p'q'}^{m})_{s+h_{p'q'}^{m}} \subset  (L_{pq}^{m})_{s-n+h_{pq}^{m}}$ 
\item[(b)]  $ \phi(w_{r}). (L_{p'q'}^{m})_{s+h_{p'q'}^{m}} \subset  (L_{pq}^{m})_{s-r+h_{pq}^{m}}   $
\end{enumerate} \end{lemma}
 \begin{definition}Let $h = 1-\lambda$ be the conformal dimension of $\phi$, \\Êand  $\triangle = 1- \lambda + \mu = h+ h_{p'q'}^{m} - h_{pq}^{m} $ ; we define:  \begin{center}  $\phi(z)=  \sum_{n \in \ZZZ +\frac{\sigma}{2}} \phi(v_{n})z^{-n-\triangle} $   and \hspace{0,21cm} $\theta(z) =   \sum_{n\in \ZZZ + \frac{1-\sigma}{2} } \phi(w_{n})z^{-n-1/2-\triangle} $  \end{center}
$\phi(z)$ is called the ordinary part and  $\theta (z) = [G_{-1/2} , \phi(z)] $, the super part of the primary field.   \end{definition}
                
  \begin{lemma} \label{relationsss} (Covariance relations).

 \begin{description}
 \item[(a)]  $[L_{n} , \phi(z)] = [z^{n+1}\frac{d}{dz} +h(n+1)z^{n}] \phi(z)$ 
 \item[(b)]  $[G_{n-1/2} , \phi(z)] = z^{n}\theta(z) $
 \item[(c)]  $[L_{n} , \theta(z)] = [z^{n+1}\frac{d}{dz} + (h+1/2)(n+1)z^{n}] \theta(z)$
 \item[(d)]  $[G_{n-1/2} , \theta(z)]_{+} = [z^{n}\frac{d}{dz} + 2hn.z^{n-1}] \phi_(z)$
    \end{description}  \end{lemma}
    \begin{proof} Direct from the definition.  \end{proof}
\begin{lemma}(Compatibility condition)
 \begin{description}
 \item[(a)]  $[L_{n} , \phi(z)] = [z^{n+1}\frac{d}{dz} + (n+1)z^{n}(1-\lambda)] \phi(z)$ 
 \item[(b)]  $[G_{r} , \phi(z)] = z^{r+1/2} [G_{-1/2} , \phi(z)]$   
  \end{description}    \end{lemma} 
\begin{proof}Immediate.            \end{proof}
\begin{lemma}   If   $\widetilde{\phi}(z)$ satisfy the  compatibility condition, then,  it gives a primary fields for the Neveu-Schwarz algebra, with $\widetilde{\theta }(z) = [G_{-1/2} , \widetilde{\phi}(z)] $ as super part.   \end{lemma}
\begin{proof}It's an easy verification.       \end{proof}
\begin{proposition} (Initial term) The space of primary fields $\phi :  L_{p'q'}^{m} \otimes \F^{\sigma}_{\lambda , \mu}  \to  L_{pq}^{m}$ with every parameter fixed, is at most one-dimensional.      \end{proposition}
\begin{proof} Let $\Omega$ and $\Omega'$ be the cyclic vectors of the positive energy representations and $v \in \F^{\sigma}_{\lambda, \mu}$. Then, by the intertwining relations, $(\phi(v)\eta , \xi )$ is completly determined by the initial term $(\phi(v)\Omega , \Omega' )$. Next $(\phi(v)\Omega , \Omega' )$ is non-zero for $v$ in a subspace of $\F^{\sigma}_{\lambda, \mu}$ of at most dimension one (lemma \ref{grzgra}).   \end{proof}
\begin{remark} \label{investiggg} Using a slightly modified $(A_{n}B)(z)$ formula (see proposition \ref{investig}), we can  inductively generate many fields from a given field $\psi$. \\ÊFor example we find: 
\begin{center} $(L_{n}\psi)(z) = [\sum C_{n+1}^{r}(-z)^{r}L_{n-r}]\psi(z)  - \psi(z)[\sum C_{n+1}^{r} (-z)^{n+1-r}L_{r-1}]$    \end{center}
We can also write a formula for $G_{r}$.  Now,  we see that: 
\begin{center} $(L_{0}\phi)(z) = [L_{0} , \phi(z)] - z [L_{-1} , \phi(z)] = h\phi(z)  $  \end{center}
It's easy to see that using this machinery from $\phi(z)$ we generate the unitary $\Vir_{1/2}$-module $L(h, c_{m})$. Then, by FQS criterion, $h = h_{p''q''}^{m}$. We note the elements $\Phi(a,z)$ with $a \in L_{p''q''}^{m}$, $\phi(z)  = \Phi(\Omega_{p''q''}^{m},z)$ and if $\psi(z) = \Phi(a,z) $ then $(L_{n}\psi)(z) = \Phi(L_{n}.a,z)$. We do the same with $G_{r}$. We call $\Phi$ a general vertex operator, it generalizes the vertex operator of the section \ref{vosa} , it admits many properties, but we don't need to enter into details.
 \end{remark}

\begin{corollary} $\mu  = h_{p'q'}^{m} - h_{pq}^{m}$ and  $1-\lambda = h=h_{p''q''}^{m}$.    \end{corollary} 
\begin{definition} We note $\phi$ as $\phi_{pqp'q'}^{p''q''m}  $,   $\triangle$ as $  \triangle_{pqp'q'}^{p''q''m} = h_{p''q''}^{m} - h_{p'q'}^{m} + h_{pq}^{m}$.   \end{definition}
\begin{definition} With $p'' = 2k+1$ and $q'' = 2k'+1$, we call $\phi$  a primary field of charge $(k,k')$.   \end{definition}
Note that the charge and the spaces between which the field acts fixes  $\lambda$ and $\mu$, but $\sigma$ can be $0$ or $1$. Now, $\sigma = 0$ or $1$ corresponds to $\phi(z)$ with integers or half-integers modes respectively.  On our work, we only need to consider primary fields of charge $\alpha = (1/2 , 1/2) $ and   $\beta = (0,1)$:

\begin{proposition}   \label{gloriato} Up to a multiplication by a rational power of $z$:  \\
 (a)   The compression of  complex fermions gives  primary fields of charge $\alpha$. \\
 (b)  The compression of  real fermions gives  primary fields of charge $\beta$.     \end{proposition}
\begin{proof} We just check the compatibility condition using the explicit formula of GKO for  $G_{r}$. The calculation can also be made on the vertex algebra of the fermions.       \end{proof}

\subsection{Constructible primary fields and braiding  for $\Vir_{1/2}$}
 \begin{lemma}    \label{trianglu} Let $m = \ell  +2$. and \scriptsize{Ê  $ \left\{ \begin{array}{lll} p=2i+1 &  p'=2j+1 &  p''=2k+1\\ q=2i'+1& q'=2j'+1 & q''=2k'+1\end{array} \right.$}  \normalsize  
 \begin{enumerate}
 \item[(a)] $ h_{i}^{\ell}  = h_{pq}^{m} + h_{i'}^{\ell + 2} - \1/2 (i-i')^{2} $
 \item[(b)]  $\triangle_{ij}^{k \ell}  =  \triangle_{pqp'q'}^{p''q''m} +  \triangle_{i'j'}^{k' \ell  +2}  -  C_{ii'jj'}^{kk'} $  
 \end{enumerate}
  with  $C_{ii'jj'}^{kk'} = \1/2[( i-i' )^{2}-( j-j' )^{2}+( k-k' )^{2} ]$
  \end{lemma}
 \begin{proof} \begin{center} $h_{pq}^{m} = \frac{[(m+2)p-mq]^{2}-4}{8m(m+2)} = \frac{2p^{2}(m+2) - 2q^{2}m-4}{8m(m+2)}+\frac{(p-q)^{2}}{8} = h_{i}^{\ell} -h_{i'}^{\ell + 2} + \1/2 (i-i')^{2} $ \end{center}Next,  (b) is  immediate by (a).    \end{proof}
 
\begin{notation} We note $h_{ii'}^{\ell}$ , $L_{ii'}^{\ell}$, $\phi_{ii'jj'}^{kk' \ell}$ and $ \triangle_{ii'jj'}^{kk' \ell}$ instead of  $h_{pq}^{m}$, $L_{pq}^{m}$, $\phi_{pqp'q'}^{p''q'' m}$  and $ \triangle_{pqp'q'}^{p''q''m}$. \end{notation}

  \begin{definition} A non-zero primary field of charge $\alpha = (1/2 , 1/2)$ or $\beta = (0,1)$ is called constructible if it's a compression  fermions. \end{definition}
\begin{theorem} \label{woker}  (Constructible primary fields) \\
(1)$\phi_{ii'jj'}^{\alpha  \ell}$ is constructible iff: \\
 $i+i'+ 1/2 \le \ell $ and $j+j'+ 1/2  \le \ell + 2$  \\
  (a) If $\sigma = 0$: $i' = i \pm 1/2 $ and $j' = j \pm 1/2 $,  \\
  (b) If $\sigma = 1$: $i' = i \pm 1/2 $ and $j' = j \mp 1/2 $.  \\ \\ 
(2) $\phi_{ii'jj'}^{\beta \ell}$ is constructible iff: \\
  $i+i'\le \ell $ and $j+j'+ 1  \le \ell + 2$  \\
  (a) If $\sigma = 0$: $i' = i $ and $j' =j \pm 1 $.\\Ê
  (b) If $\sigma = 1$: $i' = i $ and $j' =j  $    \\ \\
  with the restriction that  $ 0 \le i,i' \le \ell /2$ and $ 0 \le j,j' \le (\ell+2) /2$.  \end{theorem}
  This section is devoted to prove the theorem.
  
\begin{remark} If we ignore $\sigma$, we see that the dimension of the spaces of constructible primary fields are $0$, $1$ or $2$-dimensional, and it's correspond to the fusion rules obtained  below.    \end{remark}
\begin{remark} The dimension of the space of all the primary fields (non necessarily construcible as above) have been calculated by Iohara and Koga \cite{ik}, using the action on the singular vectors of $\F^{\sigma}_{\lambda , \mu}$. Their result shows that in the previous cases, every primary fields are constructibles. \end{remark}
\begin{corollary}  Let $\phi$ of charge $\alpha$ or $\beta$, then $\phi \ne 0$ iff $\phi$ constructible. \end{corollary}
\begin{reminder} (GKO construction, see \cite{NSOAII} section 5)
\begin{center}  $\F_{NS}^{\gg}  \otimes L_{i}^{\ell}= \bigoplus L_{ij}^{\ell} \otimes  L_{j}^{\ell +2}$ \end{center}
 and $ \F_{NS}^{\gg} = L_{0}^{2}  \oplus L_{1}^{2}$  as $L\gg$-module.  \end{reminder}
\begin{corollary}(Braiding relations)   \label{joker}
 \\ Let $(\gamma_{1} , \gamma_{2}) = (\alpha,  \alpha)$,  $(\beta, \alpha) $ or  $(\alpha , \beta)$:
\begin{center}
 $\phi_{ii'jj'}^{\gamma_{1} \ell}(z) \phi_{jj'kk'}^{\gamma_{2} \ell}( w) = \sum  \mu_{rr'}  \phi_{ii'rr'}^{\gamma_{2} \ell}( w) \phi_{rr'kk'}^{\gamma_{1} \ell}(z)$  with $\mu_{rr'}  \ne 0 $.
\end{center}
To simplify, we don't write the dependence of $\mu_{rr'}$ on the other coefficients.  
\end{corollary}
\paragraph{proof of  theorem \ref{woker} and corollary \ref{joker} } $\begin{array}{c}  \end{array}$\\Ê This proof is an adaptation of the proof of Loke \cite{loke} for $\Vir$. \\ I thank A. Wassermann to have simplified it. \\ÊLet $H_{j}^{\ell}$, $ H_{jj'}^{\ell}$  be the $L^{2}$-completion of $L_{j}^{\ell}$ and  $ L_{jj'}^{\ell}$. \\Ê
Let $\Phi(v,z) = I \otimes \phi_{ij}^{\1/2 \ell}(v,z): \F_{NS}^{\gg}  \otimes H_{j}^{\ell} Ê\to \F_{NS}^{\gg}  \otimes H_{i}^{\ell}$. By the coset construction: 
\begin{center} $\F_{NS}^{\gg}  \otimes H_{j}^{\ell} = \bigoplus H_{jj'}^{\ell} \otimes  H_{j'}^{\ell +2}$   and  $ \F_{NS}^{\gg}  \otimes H_{i}^{\ell} = \bigoplus H_{ii'}^{\ell} \otimes  H_{i'}^{\ell +2}$  \end{center}
Let $p_{i'}$, $p_{j'}$ be the projection on $H_{ii'}^{\ell} \otimes  H_{i'}^{\ell +2}$ and  $H_{jj'}^{\ell} \otimes  H_{j'}^{\ell +2}$. \\ Let $\eta \in H_{ii'}^{\ell}$, $\xi \in H_{jj'}^{\ell}$ be non-zero fixed $L_{0}$-eigenvectors.\\  Let $\phi(v,z): H_{j'}^{\ell +2} \to H_{i'}^{\ell + 2}$, defined by: 
$\forall \eta' \in H_{i'}^{\ell+2}$ and  $\forall \xi' \in H_{j'}^{\ell+2}$, \begin{center} $(p_{i'}\Phi(v,z)p_{j'} .(\xi \otimes \xi' ) , \eta \otimes \eta'     ) =  (\phi(v,z).\xi' , \eta') $. \end{center}
 Now, by compatibility condition for $LSU(2)$: 
\begin{center} $[X(n) , \Phi(v,z)] = z^{n} \Phi(X.v,z)$  and  $[L_{n} , \Phi(v,z)] = z^{n}[z\frac{d}{dz} + (n+1)h^{\ell}_{1/2}] \Phi(v,z)$     \end{center}
Now, $X(n)$ and $L_{n}$ commute with $p_{i'}$, $p_{j'}$ and $z^{r}$ with $s \in \QQQ$, then, by easy manipulation we see that, up to multiply by a rational power of $z$: 
\begin{center} $[X(n) , \phi(v,z)] = z^{n} (\phi(X.v,z))$  and  $[L_{n} ,\phi(v,z)] = z^{n}[z\frac{d}{dz} + (n+1)h^{\ell +2}_{1/2}] (\phi(v,z))$     \end{center}
By compatibility and uniqueness theorem, $\exists s \in \QQQ$ such that $z^{s}\phi(v,z)$ is the spin $1/2$ and level $\ell + 2$ primary field  $ \phi_{i'j'}^{\ell + 2}(v,z)$ (up to a multiplicative constant) of $LSU(2)$. The power $s$ can be compute using lemma \ref{trianglu}.  It follows that $p_{i'}\Phi(v,z)  p_{j'} =  \phi_{i'j'}^{\ell + 2}(v,z) \otimes \rho(z)$. Now, $h^{\ell }_{\1/2, \1/2} = h^{\ell }_{\1/2} - h^{\ell +2 }_{\1/2}$, it follows that up to multiply by a rational power of $z$: 
\begin{center} $[L_{n} ,\rho(z)] = z^{n}[z\frac{d}{dz} + (n+1)h^{\ell }_{\1/2, \1/2}] \rho(z)$  \end{center}
We verify also, using the explicit formula for $G_{r}$ that: 
\begin{center} $[G_{-1/2} ,\rho(z)] = z^{-r-1/2}[G_{r} , \rho(z) ]$  \end{center}
Finally, by compatibility condition and uniqueness, $\exists s' \in \QQQ$ such that $z^{s'}\rho(z) $,  is the charge $(1/2, 1/2)$ primary field  $\phi_{ii' jj'}^{\1/2 \1/2 , \ell}(z)$ of $\Vir_{1/2}$ between $H_{jj'}^{\ell}$ and $H_{ii'}^{\ell}$ (up to a multiplicative constant). Finally by lemma \ref{trianglu}: 
\begin{center} $p_{i'} [I \otimes \phi_{ij}^{\1/2 \ell}(v,z)] p_{j'} = C.z^{-C_{ii'jj'}^{kk'}} \phi_{i'j'}^{\ell + 2}(v,z) \otimes  \phi_{ii' jj'}^{\1/2 \1/2 , \ell}(z)  $     \end{center}
 the value of $\sigma$ follows using characterization: integer and half-integer moded.

Now, the constant $C$ is  possibly zero. So, we will prove it's non-zero for the annonced constructible fields:

If  it exists $ j'$ such that, $\Phi(v,z)  p_{j'} = 0 $ $\forall v$, then, $\forall u \in L_{jj'}^{\ell} \otimes  L_{j'}^{\ell +2}$, $\Phi(v,z)u = 0$, but, by commutation relation with $I \otimes \psi(x,r)$ and $X(n) \otimes I$, it follows by irreducibility that $u \ne 0$ is cyclic and  $\Phi(v,z)u' = 0$ $\forall u' \in  \F_{NS}^{\gg}  \otimes L_{j}^{\ell}$. Then, $\Phi(v,z) = 0$ contradiction. So, $\forall j'$, $\Phi(v,z)  p_{j'} \ne 0 $, so it exists $  i'$ such that $p_{i'}\Phi(v,z)  p_{j'} \ne 0$. By the beginning of the proof, a necessary condition for $i'$ is that $\phi_{i'j'}^{\1/2 \ell + 2}$ is a non-zero primary field of $LSU(2)$. We will prove that this condition is also sufficient. For now, we know that for this $i'$, $\phi_{ii' jj'}^{\1/2 \1/2 , \ell}$ is a non-zero primary field of $\Vir_{1/2}$. 

Now,  $\forall i'$  let $\rho_{ii' jj'}^{\1/2 \1/2 , \ell}$ a multiple (possibly zero) of  $\phi_{ii' jj'}^{\1/2 \1/2 , \ell}$, such that: 
\begin{center} $\Phi(v,z)  =  \Phi_{ij}(v,z) = \sum   \rho_{ii' jj'}^{\1/2 \1/2 , \ell}(z) \otimes \phi_{i'j'}^{\1/2, \ell + 2}(v,z) $    \end{center}
Now, $ (\Phi_{ij}(u,z)  \Phi_{jk}(v,w) \Omega_{jj'kk'} \otimes  \Omega_{j'k'}, \Omega_{jj'ii'} \otimes \Omega_{j'i'}) \\
=  \sum  (   \rho_{ii' jj'}^{\1/2 \1/2 , \ell}(z)  \rho_{jj'kk'}^{\1/2 \1/2 , \ell}(w)   \Omega_{jj'kk'} , \Omega_{jj'ii'}) . ( \phi_{i'j'}^{ \1/2, \ell + 2}(u,z) \phi_{j'k'}^{\1/2 ,\ell + 2}(v,w)  \Omega_{j'k'}, \Omega_{j'i'})$   \\ 
We can write it as a relation between reduced $4$-point function: 

\begin{center} $F_{j}(\zeta) = \sum f_{j'}(\zeta) h_{jj'}(\zeta) $    \end{center}
We return in the context of reminder \ref{braid1} and \ref{braid2}: $F_{j}$ and $f_{j'}$ are holomorphic function from $\CCC \backslash [0 , \infty[$ to $W$. Let $v_{j'} \in W$ such that  $g(\zeta) v_{j'}  = \zeta^{\mu_{j'}} f_{j'}(\zeta)$. We apply the gauge transformation $g(\zeta)^{-1} $ on the previous equality: 
\begin{center}  $g(\zeta)^{-1}  F_{j}(\zeta)  =   \sum \zeta^{-\mu_{j'}}  v_{j'}  h_{jj'}(\zeta)  $    \end{center}

It follows that $ h_{jj'}$ is  holomorphic on $\CCC \backslash [0 , \infty[$, we get a formula for it: 

\begin{center} $  h_{jj'}(\zeta) = C. \zeta^{\mu_{j'}} (g(\zeta)^{-1}  F_{j}(\zeta) ,  v_{j'}   )  $   \end{center}
with $C$ a non-zero constant.
 
 This formula gives exactly the duality for braiding discovered by Tsuchiya-Nakanishi \cite{tsunaka}.  Then by reminder \ref{braid2}, the braiding matrix for the $ \rho_{ii' jj'}^{\1/2 \1/2 , \ell}(z)$ is the product of the braiding matrix for $LSU(2)$ at spin $1/2$ and  level $\ell$, times the transposed of the inverse of the braiding matrix for $LSU(2)$ at spin $1/2$ and  level $\ell +2$. All the coefficients are non-zero. Now, suppose that  $ \rho_{ii' jj'}^{\1/2 \1/2 , \ell}(z) =0 $, with $\phi_{ij}^{\1/2, \ell} $ and $\phi_{i'j'}^{\1/2,  \ell + 2}$ constructible then: 
 \begin{center} $0 = \rho_{ii' jj'}^{\1/2 \1/2 , \ell}(z)\rho_{ jj' ii'}^{\1/2 \1/2 , \ell}(w)  = \sum \lambda_{kk'} \rho_{ii' kk'}^{\1/2 \1/2 , \ell}(w)\rho_{ kk' ii'}^{\1/2 \1/2 , \ell}(z)   $       \end{center}
 with all braiding coefficients non-zero. But as we see previously by irreducibility, the right side admits at least a non-zero term, contradiction.
 
 For the braiding between charge $(0,1)$ and charge $(1/2 , 1/2)$ primary fields, we do the same starting with the Neveu-Schwarz fermion field $\psi(u,z) \otimes I$ commuting with $ I \otimes \phi_{ij}^{\1/2 \ell}(v,w)$. We find also that every possible braiding coefficients are non-zero.  The result follows.  \textbf{End of the proof.}

\begin{remark}  As a consequence of remark \ref{dotfat}, we know that such a braiding exists for  $(\gamma_{1} , \gamma_{2}) = (\beta,  \beta)$, but we don't know if every coefficients $ \mu_{rr'}$ are non-null. \end{remark}

\begin{proposition}  The primary fields of charge  $\alpha$ or $\beta$ are bounded and identifying the $L^{2}$-completion of  $\F^{\sigma}_{\lambda , \mu}$ with $L^{2}(\SSS^{1})e^{\sigma i\theta /2} \oplus  L^{2}(\SSS^{1})e^{(1-\sigma)i\theta /2} $, we obtain $\phi(f)$ for $f \in L^{2}(\SSS^{1})e^{\sigma i\theta /2}$,  $\theta(g)$ for $g \in L^{2}(\SSS^{1})e^{(1-\sigma)i\theta /2} $ with: \begin{center} $\Vert \phi(f) \Vert \le K \Vert f \Vert_{2}$ \quad and \quad   $\Vert \theta(g) \Vert \le K' \Vert g \Vert_{2}$  \end{center} \end{proposition}
\begin{proof} The primary fields of charge $\alpha$  or $\beta$ are constructibles, and the compressions of fermions are bounded operators.    \end{proof}

   \begin{corollary}  (Localised braiding relation)  Let $f  \in L^{2}_{I}(\SSS^{1})e^{\sigma i\theta /2}$ and $g \in L^{2}_{J}(\SSS^{1})e^{\sigma i\theta /2} $, with $I$, $J$ be two disjoint proper intervals of $\SSS^{1}$. 
Using an argument of convolution (as \cite{2} p 516), we can write the following localised braiding relations: 
\begin{center}
 $\phi_{ii'jj'}^{\gamma_{1} \ell}(f) \phi_{jj'kk'}^{\gamma_{2} \ell}( g) = \sum  \mu_{rr'}  \phi_{ii'rr'}^{\gamma_{2} \ell}( e_{\lambda}g) \phi_{rr'kk'}^{\gamma_{1} \ell}(\bar{e}_{ \lambda}f)$  with $\mu_{rr'}  \ne 0 $.
\end{center}
with $e_{\lambda} = e^{i \lambda \theta} $, $\lambda = h_{ii'}^{\ell}+  h_{kk'}^{\ell} - h_{jj'}^{\ell} -  h_{rr'}^{\ell}$ and $(\gamma_{1}  , \gamma_{2}) $ as previously.
   \end{corollary}

\subsection{Application to irreducibility}
\begin{definition} Let $\M$, $\N \subset B(H)$ be von Neumann algebra, then, \\Ê $\M \vee  \N$ is the von Neumann algebra generated by $\M$ and $\N$.     \end{definition}
\begin{notation} We simply note $\phi_{ij}^{k}(f)$ for primary field of charge $k$ for $\Vir_{1/2}$; the charge $c_{m}$ is fixed and $i = 0$ significate $i=(0,0)$.   \end{notation}
\begin{proposition}  \label{chainvn} The chains of constructible primary fields of the form: 
\begin{center} $\phi_{0i_{1}}^{\alpha}(f_{1})\phi_{i_{1}i_{2}}^{\alpha}(f_{2}) ... \phi_{i_{r-1}i_{r}}^{\alpha}(f_{r})\phi_{i_{r}0}^{\alpha}(f_{r+1}) $  with $\alpha = (\1/2 , \1/2)$  and $f_{i}$ on $I$.  \end{center} are bounded operators and generate the von Neumann algebra $\N_{00}^{\ell}(I)$.   \end{proposition}
\begin{proof} By corollary \ref{gloria} and proposition \ref{gloriato}. \end{proof}
\begin{remark}  \label{joria}Let $\sigma_{t} $ be the geometric modular action described on reminder \ref{complexfer}. Let $\psi^{k}_{ij}(f) $ be a bounded primary field of charge $k$ concentrated on a proper interval $J$. Let $\sigma_{t}(\psi^{\alpha}_{ij}(f) ):= \pi_{i}(\varphi_{t})\psi^{\alpha}_{ij}(f)\pi_{j}(\varphi_{t})^{\star}=\psi^{\alpha}_{ij}(u_{t}.f) $ by  the covariance relations. Then, $\sigma_{t}(\psi^{\alpha}_{ij}(f) ) $ is a primary field concentrated on $\varphi_{t}(J ) \to \{ 1  \} $ (when $t \to \infty$ ).     \end{remark}
\begin{reminder} \label{cancell} (Cancellation theorem)   If a unitary representation of a connected semisimple non-compact group with finite center has no fixed vectors, then its matrix coefficients vanish at $\infty$. We can find a proof on Zimmer's book \cite{zimm}. For example, $G=SU(1,1) \simeq SL(2, \RRR)$ (non-compact) is implemented on the irreducible positive energy representations $H$ of $\Vir_{1/2}$, which give a unitary representation of a central cyclic extension $\G$ of $G$, whose Lie algebra is generated by $L_{-1}$, $L_{0}$ and $L_{1}$. But if $ \xi \in H$,  $L_{0}\xi = 0$ implies immediatly that $H = H_{0}$ and  $\xi = \Omega$ (up to a multiplicative constant). So $G$ admits no fixed vectors outside of the vacuum. But the modular operators $U_{t}$ go to $\infty$ when $t \to \infty$. Then, their matrix coefficients vanish at $\infty$. In our case, we can prove the cancellation theorem directly, because $H$ decomposes into a direct sum of  irreducible positive energy representation of $\G$ and each summands is a discrete series representation of $\G$, so can be realized as a subrepresentation of $L^{2}(\G)$, and then has matrix coefficient tending to zero at $\infty$ (see Pukanszky \cite{puka}).  \end{reminder}

\begin{proposition} (Generically non-zero) \label{generically}  Let $a = \phi^{\alpha}_{\alpha 0}(f) $ and $b =  \phi^{\alpha}_{0 \alpha}(g) $ with $f$, $g$  on  proper intervals. Then, $(b a \Omega , \Omega)$ is  non-zero in general.    \end{proposition}
\begin{proof} Let  $a = \phi^{\alpha}_{\alpha 0}(f)  \ne 0$  and   $R_{\theta}$ be the quantized rotation action: $R_{\theta} = e^{iL_{0}\theta}$ (see remark \ref{realcova}). Let  $b_{\theta} = R^{\star}_{\theta} a^{\star} R_{\theta}$. We suppose that $(b_{\theta} a \Omega , \Omega ) = 0$ for $\vert \theta - \theta_{1} \vert \le \varepsilon $ with $\theta_{1}$ fixed and $\varepsilon  > 0 $.  Then $(R^{\star}_{\theta} a^{\star} R_{\theta} a \Omega , \Omega ) = 0$. But $L_{0} \Omega = 0$ on the vacuum representation. Then, $R_{\theta}\Omega = \Omega$ and $(  R_{\theta} a \Omega , a \Omega ) = 0$. Now, by positive energy of the representation $a\Omega = \sum_{n \in \1/2 \NNN} \xi_{n} $ (coming from the orthogonal decomposition for $L_{0}$) and $\Vert a \Omega \Vert ^{2} = \sum \Vert \xi_{n} \Vert ^{2}$. Now, with $z = e^{i \theta /2}$, $(  R_{\theta} a \Omega , a \Omega )  = \sum_{n \in \NNN} z^{n}  \Vert \xi_{n/2} \Vert ^{2} = f(z)$, let $g(z) = f(e^{-i \theta_{1} /2} z) $. Then, $g$ extends to a continuous function on the closed unit disc, holomorphic in the interior and vanishing on the unit circle near $\{  1  \}$. By the Schwarz reflection principle and the  Cayley transfrom, $g$ must vanishes identically in $z$. So,  $(  R_{0} a \Omega , a \Omega ) = \Vert a \Omega \Vert ^{2}= 0$. Then $a\Omega = 0 $, so  $a^{\star}a \Omega= 0$. But $\Omega$ is a separating vector on the von Neumann algebra, so $a^{\star}a = 0$, and $a = 0$, contradiction.  
\end{proof}

\begin{proposition}(Leading term in OPE of primary fields)  \\ÊLet $I$ be a proper interval of $\SSS^{1}$, and $I_{1}$, $I_{2}$ be subintervals be subintervals obtained by removing a point. Let $a_{ \nu \mu }$ and  $b_{\mu \nu }$ be non-zero primary fields of charge $\alpha$,  localised in $I_{1}$ and $I_{2}$ respectively, then $\sigma_{t}(a_{ \nu \mu }b_{\mu \nu }) \to^{w} Id_{H_{i}}  $ (up to a multiplicative constant).   \end{proposition}
\begin{proof} We adapt to $\Vir_{1/2}$,  a proof of A. Wassermann \cite{wass3} for $LSU(2)$. 

Without a loose of generality, we  can take   $ \{ 1 \} \in  \bar{ I_{1} } \cap \bar{ I_{2} } $. Let $a$ and $b$ be generic primary fields of charge $\alpha$ concentrate on $I_{2} $ and $I_{1} $ respectively.

(1)We first prove that $\sigma_{t}(a_{ 0 \alpha }b_{\alpha 0 }) \to^{w} C$ non-zero constant:     \\  
$\Vert \sigma_{t}(a_{ 0 \alpha }b_{\alpha 0 }) \Vert $ is clearly bounded, then by the weak compacity of the unit ball, it exists a sequence $t_{n}$ such that  $\sigma_{t_{n}}(a_{ 0 \alpha }b_{\alpha 0 }) \to^{w} T  $. By the remark \ref{joria},  $\sigma_{t_{n}}(b_{ 0 \alpha }a_{\alpha 0 })$ is concentrated  on $J_{n} $ with $\bigcap J_{n} =  \{ 1 \}$. We obtain that $T$ supercommutes with $\bigvee \N_{00}^{\ell}(J_{n}^{c})$. By Araki-Haag duality, $(\bigvee \N_{00}^{\ell}(J_{n}^{c}))^{\natural} = \bigcap \N_{00}^{\ell}(J_{n})) =  \N_{00}^{\ell}(\{ 1 \}) = \CCC$. Then $T \in \CCC Id$. Now, $ ( \sigma_{t_{n}}(a_{ 0 \alpha }b_{\alpha 0 })\Omega , \Omega ) =  ( a_{ 0 \alpha }b_{\alpha 0 } \Omega , \Omega ) $ because $\pi_{0}(U_{t})\Omega = \Omega$ (see remark \ref{realcova}). Now $( a_{ 0 \alpha }b_{\alpha 0 } \Omega , \Omega )  = k $ generically non-zero (proposition \ref{generically}) and $T = kId$. Now, $k$ is independant on the sequence $(t_{n})$, so  $\sigma_{t}(a_{ 0 \alpha }b_{\alpha 0 }) \to^{w} k.Id \ne 0$.

(2) We now prove that $\sigma_{t}(a_{ \gamma \alpha }b_{\alpha 0 }) \to^{w}  0$ if $\gamma \ne 0$.      \\  
Idem, it exists a sequence $t_{n}$ such that  $X_{n} = \sigma_{t_{n}}(a_{ \gamma \alpha }b_{\alpha 0 }) \to^{w} T  $. 
Let $\xi$ be a finite energy vector in $H_{\gamma}$, then $(X_{n} \Omega , \xi) = (\pi_{\gamma}(U_{t_{n}})a_{ \gamma \alpha }b_{\alpha 0 } \Omega , \xi ) = ((\pi_{\gamma}(U_{t_{n}}) \eta , \xi) \to 0$ when $t_{n} \to \infty$ by the cancellation theorem (reminder \ref{cancell})Then, $T\Omega = 0 $, so $T^{\star}T\Omega = 0$. But $\Omega$ is a separating vector on the von Neumann algebra, so $T^{\star}T= 0$ and $T = 0$.
Now, the $0$ is independent of the choice of the sequence, then: $\sigma_{t}(a_{ \gamma \alpha }b_{\alpha 0 }) \to^{w}  0$.

 (3) We  prove that if $a_{\nu \mu} \ne 0$, then  $\sigma_{t}(a_{ \nu \mu }b_{\mu \nu }) \to^{w} C' $ non-zero constant: \\
Idem, it exists a sequence such that $\sigma_{t_{n}}(a_{ \nu \mu }b_{\mu \nu }) \to^{w} R  $. Now, let $y_{\nu 0 } = x_{\nu \lambda_{1}}x_{\lambda_{1} \lambda_{2}}...x_{\lambda_{r} 0} $ be a chain between $\nu$ and $0$ with the minimal number  of primary fields of charge $\alpha$, concentrate  on a proper closed $K$ interval out of $\{ 1 \}$. Then for $t$ sufficiently large, we can apply the braiding formulas on $\sigma_{t}(a_{ \nu \mu }b_{\mu \nu }) y_{\nu 0 } $. We obtain necessarily  $\sigma_{t}(a_{ \nu \mu }b_{\mu \nu }) y_{\nu 0 }  = \sum_{\gamma \ne 0} A_{\gamma}\sigma_{t}(a_{ \gamma \alpha }b_{\alpha 0 })  + \lambda y_{\nu 0 }  \sigma_{t}(a_{ 0 \alpha }b_{\alpha 0 })$, with $\lambda \ne 0$, $A_{\gamma}$ a linear sum of non-minimal chains between $\nu$ and $\gamma$ (note that in general,  there are many ways to go between $0$ and $\nu$ minimally, but by the structure of the braiding rules, only the way chosen for $ y_{\nu 0 }$ can appear at the end). Now, by (1) and (2), the previous equality (with $t = t_{n}$) weakly converge to $R y_{\nu 0 }  =  \lambda   y_{\nu 0 } C = \lambda C y_{\nu 0 }    $ with $\lambda C$ a non-zero constant.  Now, $R \in \N^{\ell}_{\nu}(K^{c})$, then $R y_{\nu 0 }  = y_{\nu 0 }  \pi_{0}(R) = \lambda Cy_{\nu 0 } $. Now $\sigma_{t}( y_{\nu 0 })$ is also a minimal chain  of charge $\alpha$ between $\nu$ and $0$, concentrate  on a proper closed interval out of $\{ 1 \}$, so $\sigma_{t}( y_{\nu 0 }) \pi_{0}(R) = C' \sigma_{t}( y_{\nu 0 })$ with $C'$ a non-zero constant.  Then $\sigma_{t}( y_{\nu 0 })^{\star}\sigma_{t}( y_{\nu 0 }) \pi_{0}(R) = C' \sigma_{t}( y_{\nu 0 })^{\star} \sigma_{t}( y_{\nu 0 })$. But $\sigma_{t}( y_{\nu 0 })^{\star}\sigma_{t}( y_{\nu 0 }) = \sigma_{t}( y_{\nu 0 }^{\star} y_{\nu 0 }) \to_{w} k.Id \ne 0$  as for (1). So $\pi_{0}(R) = C' =R$.
 \end{proof}
\begin{proposition}  (von Neumann density) Let $I$ be a proper interval of $\SSS^{1}$, and $I_{1}$, $I_{2}$ be subintervals such that $I = I_{1} \cup I_{2}$.
\begin{center} $\N_{ij}^{\ell} (I_{1}) \vee \N_{ij}^{\ell}(I_{2}) = \N_{ij}^{\ell} (I)$.    \end{center}  \end{proposition}
\begin{proof} By the local equivalence for $\Vir_{1/2}$ (see section \ref{richarder}), we only need to prove the result on the vacuum. By proposition \ref{chainvn} we only need to work with chains.  Consider the chain $\phi_{0i_{1}}^{\alpha}(f_{1})\phi_{i_{1}i_{2}}^{\alpha}(f_{2}) ... \phi_{i_{r-1}i_{r}}^{\alpha}(f_{r})\phi_{i_{r}0}^{\alpha}(f_{r+1}) \in  \N_{00}^{\ell}(I)$, with $f_{k} \in  L^{2}_{I}(\SSS^{1})$. Now,  $f_{k} = f_{k}^{(1)} + f_{k}^{(2)}$, with $ f_{k}^{(i)} $ concentrated on $I_{i} $. Now, a primary field $\phi_{ij}^{k}(f)$ is linear in $f$, so, we can develop the chain into a sum of chains of primary filed localized exclusively on $I_{1}$ or $I_{2}$. Next, applying the braiding relations, we can obtain a linear combination of chains, on which the  primary field localized on $I_{1}$ and $I_{2}$ are separated; generically of the form:  
\begin{center} $\phi_{0j_{1}}^{\alpha}(g_{1})\phi_{j_{1}j_{2}}^{\alpha}(g_{2}) ...  \phi^{\alpha}_{j_{s-1} j_{s} } (g_{s-1})  \phi^{\alpha}_{j_{s} j_{s+1} } (h_{s+1})   ...\phi_{j_{r-1}j_{r}}^{\alpha}(h_{r})\phi_{j_{r}0}^{\alpha}(h_{r+1}) $  \end{center}
with $g_{k}$ and $ h_{k}$  concentrate on  $I_{1}$ and $I_{2}$ respectively.  Now, if $j_{s} = 0$, then, the previous chain is a product $a.b$ with $a \in \N_{11}^{m} (I_{1}) $ and $b \in \N_{11}^{m} (I_{2})$.  \\ 
Else, if $j_{s} \ne 0$, using the previous proposition step by step, we see that the chain is the weak limit of chains with $0$ on the middle, the result follows.
  \end{proof}
  \begin{lemma}(Covering lemma) Let $(I_{n})$ be a covering of $\SSS^{1}$ by open proper intervals. Then $\Vir_{1/2}(\SSS^{1})$ is the linear span of the $ \Vir_{1/2}(I_{n})$.  And so $\bigvee \pi(  \Vir_{1/2}(I_{n}))'' = \pi(  \Vir_{1/2}(\SSS^{1}))'' = B(H)  $.  \end{lemma}
  \begin{proof} With a partition of the unity.     \end{proof}
  
  \begin{theorem} Let $I$ be a proper interval of $\SSS^{1}$, then, the Jones-Wassermann subfactorÊ $\N_{ij}^{\ell}(I) \subset \N_{ij}^{\ell}(I)^{\natural}$ is irreducible, i.e. $\N_{ij}^{\ell}(I)^{\natural} \cap \N_{ij}^{\ell}(I^{c})^{\natural}  = \CCC$.   \end{theorem}
 \begin{proof}  Let $I_{1}$, $I_{2}$ be  two proper subintervals of $I$ obtained by removing a point. Let $J_{1} = I$, $J_{2} = \overline{I_{1} \cup I^{c}}$ and $J_{3} = \overline{ I^{c} \cup I_{2}}$. 
 Let $\M = \N_{ij}^{\ell}(I) \vee\N_{ij}^{\ell}(I^{c}) $, then $ \N_{ij}^{\ell}(I), \N_{ij}^{\ell}(I^{c}), \N_{ij}^{\ell}(I_{1})$ and $ \N_{ij}^{\ell}(I_{2})   \subset \M$. By von Neumann density, 
 $\N_{ij}^{\ell}(J_{2}) = \N_{ij}^{\ell}(I_{1}) \vee\N_{ij}^{\ell}(I^{c}) \subset \M $, and idem $\N_{ij}^{\ell}(J_{3}) \subset \M$. Let $K_{1}$, $K_{2}$, $K_{3}$ be open subintervals of $J_{1}$, $J_{2}$ and  $J_{3}$ such that $K_{1} \cup K_{2} \cup K_{3} = \SSS^{1}$. Now, $N_{ij}^{\ell}(K_{1}) \vee N_{ij}^{\ell}(K_{2}) \vee N_{ij}^{\ell}(K_{3}) \subset \M$, but $N_{ij}^{\ell}(K_{1}) \vee N_{ij}^{\ell}(K_{2}) \vee N_{ij}^{\ell}(K_{3}) = B(H_{ij}^{\ell})$ by covering lemma. So $\M = B(H_{ij}^{\ell})$ and $\CCC = \M^{\natural} = \N_{ij}^{\ell}(I)^{\natural} \cap \N_{ij}^{\ell}(I^{c})^{\natural}$.  \end{proof}

   \newpage
\section{Connes fusion and subfactors}
\subsection{Recall on subfactors} See the book \cite{js} for a complete introduction to subfactors.  
\begin{definition} Let $\M$ and $\N$ be von Neumann algebra, then, an inclusion $\N \subset \M$ is called a subfactor.   \end{definition}
\begin{reminder} A factor $\M$ of type II admits a canonical trace $tr$. The image of $tr$ on the subset of projection of $\M$ is $[0,1]$ or $[0, \infty]$. \\Then, $\M$ is said to be a factor of type II$_{1}$  or  II$_{\infty}$.  \end{reminder}
\begin{reminder} (Basic construction)   Let the subfactor $\N \subset \M$, with $\M$ and $\N$ II$_{1}$ factors. Let $tr$ be the trace on $\M$, then, it admit the following inner product: $(x,y) := tr(xy^{\star})$. Let $H = L^{2}(\M , tr)$ and $ L^{2}(\N , tr )$ be the $L^{2}$-completions of $\M$ and $\N$. Let $e_{\N}$ be the orthogonal projection of $L^{2}(\M , tr)$ onto $L^{2}(\N , tr)$. \\ Let $\langle \M , e_{\N} \rangle = ( \M \cup \{ e_{\N} \})''   \subset B(H)$. It admit a trace called $tr_{\langle \M , e_{\N} \rangle}$. \\ The tower $\N \subset \M \subset \langle \M , e_{\N} \rangle$ is called the basic construction.   \end{reminder}
\begin{reminder} (Index of subfactors)  \  Let the previous subfactor $\N \subset \M$.  \\ Then we can define its index $[\M : \N] = (tr_{\langle \M , e_{\N} \rangle}(e_{\N}))^{-1} \in [1 , \infty]$. \\ÊThe index admits another definition as the von Neumann dimension (see \cite{js}) of the $\N$-module $H = L^{2}(\M , tr)$, ie $[\M : \N]= dim_{\N}(H)$.    \end{reminder}
\begin{reminder}   (Jones' theorem, see \cite{23}) Every possible index of II$_{1}$-subfactors: 
  \begin{center}  $\{4cos^{2}(\frac{\pi}{m}) \vert m = 3, 4, ... \} \cup [ 4 , \infty ] $   \end{center} 
  In the continuation of the basic construction, we can build a graph from a subfactor, called its principal graph. If the subfactor admits a finite index then the square of the norm of the matrix of its principal graph is exactly the index. Now, this matrix admits only integers values, and a theorem of Kronecker said that the norm of an integer valued matrix is in $\{2cos(\frac{\pi}{m}) \vert m = 3, 4, ... \} \cup [ 2 , \infty ] $. Finally, it's proved that every possible such norms are realized from subfactors.    \end{reminder}
  \begin{definition} A subfactor of finite index $\M \subset \N$ is said to be irreducible if either of the following equivalent conditions are satisfied:
  \begin{enumerate}
  \item[(a)] $L^{2}(\M)$ is irreducible as an $\N$-$\M$-bimodule.
  \item[(b)] The relative commutant $\N' \cap \M$ is $\CCC$.
  \end{enumerate}
  
  \end{definition}

\subsection{Bimodules and Connes fusion}
 \begin{definition}If $\M$, $\N$ are $\ZZZ_{2}$-graded von Neumann algebra, a $\ZZZ_{2}$-graded Hilbert space $H$ is said to be a $\M$-$\N$-bimodule if: 
    \begin{enumerate}
    \item[(a)] $H$ is a left $\M$-module.
        \item[(b)] $H$ is a right $\N$-module.
        \item[(c)] the action of $\M$ and $\N$ supercommute; i.e., \\ $\forall m \in \M$, $ n \in \N$, $\xi \in H$, $(m.\xi).n =(-1)^{\partial m \partial n} m.(\xi.n)$.
    \end{enumerate}   \end{definition}
    \begin{definition} Let $\Omega \in H_{0}$ be a vacuum vector, then $H_{0}$ is a $\M$-$\M$ bimodule,
  because by Tomita-Takesaki theory, $J \M J = \M'$, by lemma \ref{naturals}, $\M^{\natural} = \kappa \M' \kappa^{\star} \simeq \M'  \simeq \M^{opp}$. Now, $y^{\star} x^{\star} = (xy)^{\star}$ and  $\M^{opp}$ is the opposite algebra: $a \times b = b.a$.  Then $x.(\xi .y) := x  (\kappa J y^{\star} J  \kappa^{\star})  \xi$ gives the bimodule action.\end{definition}

  \begin{definition} (Intertwinning operators) \ Let   $X$, $Y$ be $\ZZZ_{2}$-graded $\M$-$\M$ bimodules,  $\X = Hom_{-\M}(H_{0} , X )$ and $\Y = Hom_{\M-}(H_{0} , Y )$ be the space of  bounded operators that superintertwin the left (resp. the right) action of $\M$.     \end{definition}
\begin{lemma}Consider the algebraic  tensor product  $\X \otimes \Y$, we define a pre-inner product  by: 
\begin{center} $(x_{1} \otimes y_{1} , x_{2} \otimes y_{2}) =(-1)^{(\partial x_{1} + \partial x_{2}) \partial y_{2} } ( x_{2}^{\star}x_{1}y_{2}^{\star}y_{1} \Omega , \Omega ) $  \end{center}
\end{lemma}
\begin{proof} As for \cite{2} p 525-526.   \end{proof}
\begin{definition}
The $L^{2}$-completion is  called the Connes fusion between $X$ and $Y$, and noted $X  \boxtimes Y$, naturally a  $\ZZZ_{2}$-graded  $\M$-$\M$ bimodule. \end{definition}

\begin{lemma} There are canonical unitary isomorphism \begin{center}$H_{0} \boxtimes X \simeq X \simeq X \boxtimes H_{0}$.  \end{center}   \end{lemma}
\begin{proof} If $Y = H_{0}$, the unitary $X \boxtimes H_{0} \to X$ is given by $x \otimes y \mapsto xy \Omega$, and the unitary $H_{0} \boxtimes X \to X$ is given by $y \otimes x \mapsto (-1)^{\partial x \partial y} xy\Omega$.    \end{proof}
\begin{lemma} \label{density45} $\X$ can be seen as a dense subspace of $X$  via $x \leftrightarrow x \Omega $. \end{lemma}
\begin{proof} $\X = \X. \pi_{0}(\M(I^{c}))$, so by Reeh-Schlieder $\X \Omega$ is dense in $\X H_{0}$. \\ Now, $\X H_{0} =[\pi_{X}(\M(I^{c})) \X]. [\pi_{0}(\M(I)). H_{0}] = \pi_{X}(\M(I^{c}).\M(I))  \X \H_{0} \\Ê= \pi_{X}(\langle \M(I^{c}).\M(I) \rangle_{lin})  \X \H_{0}$. But, because  $\M(I^{c})$ and  $\M(I)$ supercommute, the $\star$-algebra generated by $\M(I^{c}).\M(I)$ is exactly its linear span, then, $\pi_{X}(\langle \M(I^{c}).\M(I) \rangle_{lin})$ is weakly dense in $ \pi_{X}( \M(I^{c}).\M(I) )''$. So, by von Neumann density $\X H_{0}$ is dense in $ \bigoplus B(H_{i}) \X H_{0} = X$, with $X = \bigoplus H_{i}$. \end{proof}
\begin{lemma}  (Hilbert space continuity lemma)\\ÊThe natural map $\X \otimes \Y \to  X \boxtimes Y$ extends canonically to continuous maps $X \otimes \Y \to  X \boxtimes Y$ and $\X \otimes Y \to  X \boxtimes Y$. In fact $\Vert x_{i}\otimes y_{i}  \Vert^{2} \le \Vert x_{i}x_{i}^{\star}  \Vert \sum \Vert  y_{i}\Omega \Vert^{2}$ and $\Vert x_{i}\otimes y_{i}  \Vert^{2} \le \Vert y_{i}y_{i}^{\star}  \Vert \sum \Vert  x_{i}\Omega \Vert^{2}$  \end{lemma}
\begin{proof} As for \cite{2} p 526.  \end{proof} 
\begin{lemma}  $ \boxtimes$ is associative.  \end{lemma}
\begin{proof} As for \cite{2} p 527.  \end{proof}

 \subsection{Connes fusion with $H_{\alpha}$ on $\Vir_{1/2}$} \label{weakalpha}
 \begin{remark}Note that the primary fields $\phi$ we consider are always the ordinary part and so even operators. In fact, we only need to consider even intertwiner operators because each odd intertwiner operator is the product of an even one and an odd operator on the vacuum local von Neumann algebra.\end{remark}
 \begin{definition} Let  $ \langle  i , j \rangle := \{  k  \hspace{0,2cm} \vert  \hspace{0,2cm} \phi_{ij}^{k} \ne 0$ \}.   \end{definition}
 Recall that the primary field of charge $\alpha = (1/2,1/2)$ are bounded.  Let the graph $\G_{\alpha}$ with vertices $\{ i \}$ and an edge between $i$ and $j$ if $j \in \langle  \alpha , i \rangle$ ; then,  $\alpha$ is a weak generator in the sense that the graph $\G_{\alpha}$ is connected.  Let $I$ be a non-trivial interval of $\SSS^{1}$, and let  $f$ and $g$  be $L^{2}$-functions localized in $I$ and $I^{c}$ respectively. Recall that every possible braiding at charge $\alpha$ admits non-null coefficients, ie; 
 $ \phi_{ij}^{\alpha}(z) \phi_{jk}^{\alpha}(w)   = \sum \lambda_{l}  \phi_{il}^{\alpha}(w) \phi_{lk}^{\alpha}(z)$   with  $\lambda_{l}  \ne 0$ iff  $l \in \langle  \alpha , i \rangle \cap \langle \alpha , k \rangle$. Then, by the standard convolution argument:  $ \phi_{ij}^{\alpha}(f) \phi_{jk}^{\alpha}(g)   = \sum \lambda_{l}  \phi_{il}^{\alpha}(e_{l}g) \phi_{lk}^{\alpha}(\bar{e_{l}}f)$   with  $e_{l}$ the phase correction.  
We note $a_{0 \alpha} =\phi_{0 \alpha }^{\alpha}(f)$, $b_{ \alpha 0} = \phi_{ \alpha 0}^{\alpha}(g) $ called the principal part. We define the non-principal parts $a_{ij}$ and $b_{ij}$ such that they incorporate the phase correction in the braiding relations. Next, if $a_{ij} = \phi_{ij}^{\alpha}(h)$ then $ a_{ij}^{\star} =  \phi_{ji}^{\alpha}(\bar{h})$, so we note $\bar{a}_{ji} = a_{ij}^{\star} $: 
\begin{corollary}  (Braiding relations) 
\begin{enumerate}
\item[] $ b_{ij} a_{jk}  = \sum \nu_{l}  a_{il} b_{lk}$ \quad with  $\nu_{l}  \ne 0$ iff  $l \in \langle  \alpha , i \rangle \cap \langle \alpha , k \rangle$      
\end{enumerate}  \end{corollary}
\begin{corollary} (Abelian braiding) \ If  $\#( \langle  \alpha , i \rangle \cap \langle \alpha , k \rangle )  = 1$ then: 
\begin{enumerate}
\item[] $ b_{ij} a_{jk}  =\nu  a_{ij} b_{jk}$ \quad with  $\nu  \ne 0$   
\end{enumerate}   \end{corollary}
\begin{lemma} \label{density} The set of vectors of the form $\eta  = (\eta_{i})$ with, $\eta_{i} = \pi_{i}(x)b_{ij} \xi$, $i \in  \langle  \alpha , j \rangle$,   $x \in \M(I^{c})$ and $\xi \in H_{j}$, spans a dense subspace of $\bigoplus H_{i}$.  \end{lemma}
\begin{proof} By Reeh-Schlieder, choosing a non-null vector $v_{j} \in F_{j}$,  $\pi_{j}(\M(I^{c}))v_{j}$ is dense in $H_{j}$. Now, by intertwining, $b_{ij}\pi_{j}(\M(I)) = \pi_{i}(\M(I))b_{ij}$. Then, if  $b_{ij}  v_{j} = 0 $, then, $b_{ij} $ vanishes on a dense subspace, and so by continuity, $b_{ij} =0$, contradiction. So, $b_{ij}  v_{j} \ne 0 $.  Now, clearly, the set of vector $\rho = (\rho_{i})$, with    $\rho_{i} = \pi_{i}(x)b_{ij} \pi_{j}(y) v_{j}$,  $x \in \M(I^{c})$ and $y \in \M(I)$, is a subset of the set of the lemma. Now, by intertwining  $\rho_{i} = \pi_{i}(x) \pi_{i}(y)b_{ij} v_{j}$. Let $\pi = \bigoplus \pi_{i}$ and $w = (w_{i})$, with $w_{i} = b_{ij}  v_{j} \ne 0$. Then, the set of $\rho$ is exactly $\pi(\M(I^{c}).\M(I)).w$. Next, because $\M(I^{c})$ and $\M(I)$ commute, the linear span of  $\pi(\M(I^{c}).\M(I))$ is weakly dense in $\pi(\M(I^{c}).\M(I))'' = \bigoplus B(H_{i})$ by von Neumann density. So, the set spans a dense subspace of $ ( \bigoplus B(H_{i}))w = \bigoplus H_{i} $ because $w_{i} \ne 0$.    \end{proof}
\begin{remark} $\bar{a}_{ij}.a_{ji} \in   Hom_{\M(I^{c})}(H_{i}, H_{i})  = \pi_{i}(\M(I^{c}))' $. \\
In particular,  $\bar{a}_{0 \alpha}.a_{\alpha 0} \in \pi_{0}(\M(I)) $  by Haag-Araki duality.   \end{remark}
\begin{definition} Let $\vert i \vert$ be the less number of edges from $i$ to $0$ in the connected graph $\G_{\alpha}$.  \end{definition}
\begin{theorem}  (Transport formula)    \begin{displaymath} \pi_{i} (\bar{a}_{0 \alpha}.a_{\alpha 0}) = \sum_{j \in \langle  \alpha , i \rangle} \lambda_{j} \bar{a}_{ij}.a_{ji} \quad   \textrm{with} \  \lambda_{j} > 0.  \end{displaymath}
\end{theorem}
\begin{proof}  We prove by induction on $\vert i \vert$.  We suppose that:
\begin{center} $\pi_{i} (\bar{a}_{0 \alpha}.a_{\alpha 0}) = \sum_{j \in \langle  \alpha , i \rangle} \lambda_{j} \bar{a}_{ij}.a_{ji}$ \quad  and \quadÊ $\pi_{i} (\bar{b}_{0 \alpha}.b_{\alpha 0}) = \sum_{j \in \langle  \alpha , i \rangle} \lambda'_{j} \bar{b}_{ij}.b_{ji}$ \end{center}
 $(1)$ \   Polarizing the second identity, we get:   
 \begin{center}  $\pi_{i} (\bar{b}_{0 \alpha}.b'_{\alpha 0}) = \sum_{j \in \langle  \alpha , i \rangle} \lambda'_{j} \bar{b}_{ij}.b'_{ji}$    \end{center}
 Now, with $x \in \M(I^{c})$ and $b'_{ij} = \pi_{i}(x) b_{ij}   \pi_{j}(x)^{\star}  $, we get: 
 \begin{center}  $\pi_{i} (\bar{b}_{0 \alpha}.\pi_{\alpha}(x) b_{\alpha 0} \pi_{0}(x)^{\star}) = \sum_{j \in \langle  \alpha , i \rangle} \lambda'_{j} \bar{b}_{ij}. \pi_{j}(x).b_{ji} \pi_{i}(x)^{\star}$    \end{center}
 Now, $\pi_{i}(\pi_{0}(x)^{\star}) = \pi_{i}(x)^{\star}$, so you can simplify by $ \pi_{i}(x)^{\star}$: 
 \begin{center}  $\pi_{i} (\bar{b}_{0 \alpha}.\pi_{\alpha}(x) b_{\alpha 0} ) = \sum_{j \in \langle  \alpha , i \rangle} \lambda'_{j} \bar{b}_{ij}. \pi_{j}(x).b_{ji} $    \end{center}
$(2)$ \ Next, by $(1)$ and the braiding relations, $\bar{a}_{ik} \pi_{k} (\bar{b}_{0 \alpha}\pi_{\alpha}(x)b_{\alpha 0}) a_{ki}= \\
 \pi_{i} (\bar{b}_{0 \alpha}\pi_{\alpha}(x)b_{\alpha 0}) \bar{a}_{ik} a_{ki}    =  \sum_{j} \sum_{l,s} \lambda'_{j} \nu_{l} \mu_{s} \bar{b}_{ij}\bar{a}_{jl}a_{ls}\pi_{s}(x)b_{si}$. \\
Let $y = \bar{a}_{ik} \pi_{k} (\bar{b}_{0 \alpha}\pi_{\alpha}(x^{\star }x)b_{\alpha 0}) a_{ki} =a^{\star}_{ki} \pi_{k} (b^{\star}_{ \alpha 0}\pi_{\alpha}(x^{\star }x)b_{\alpha 0}) a_{ki}$ clearly a positive operator, then, $\forall \xi \in H_{i}$, $(y \xi , \xi) \ge 0$. Then, with $\eta_{s} = \pi_{s}(x)b_{si}\xi$, we obtain:
\begin{center}  $\sum \lambda'_{j} \nu_{l} \mu_{s}  (a_{ls}\eta_{s} , a_{lj}\eta_{j} ) \ge 0 $    \end{center}
$(3)$ \ We now show that this inequality is linear in $\eta$: \\
Let $\tilde{\eta} = \sum \eta^{r}$ with $ \eta^{r} = ( \eta^{r}_{s})$,  $\eta^{r}_{s}=  \pi_{s}(x_{r})b_{si}\xi_{r}  $, $x_{r} \in  \M(I^{c})$ and $\xi_{r} \in H_{i}$. 
Idem, $Y = (y_{rt})$ with $y_{rt} = a^{\star}_{ik} \pi_{k} (b^{\star}_{ \alpha 0}\pi_{\alpha}(x_{r}^{\star }x_{t})b_{\alpha 0}) a_{ik}$,  is a positive operator-valued matrix, so that $\sum_{r,t}(y_{rt} \xi_{t} , \xi_{r}) \ge 0$, which is exactly the inequality $ \sum \lambda'_{j} \nu_{l} \mu_{s}  (a_{ls}\tilde{\eta}_{s} , a_{lj}\tilde{\eta}_{j} ) \ge 0 $, and the linearity follows.\\
$(4)$ \ Next, by lemma \ref{density}, the set of such $\eta$ span a dense subspace of $\bigoplus H_{s}$, then, by linearity and continuity, the inequality runs $\forall \eta \in \bigoplus H_{s}$. \\ In particular, taking all but one $\eta_{j}$ equal to zero, we obtain $\forall  \eta_{j} \in H_{j}$: \begin{center}   $\lambda'_{j} \mu_{j} \sum_{l} \nu_{l}  \Vert a_{lj}\eta_{j} \Vert^{2} \ge 0 $ \end{center}
$(5)$ \ Now, restarting from  $\tilde{Y} = (\pi_{k}(z_{u})^{\star} Y \pi_{k}( z_{v})) $ with $z_{u} \in  \M(I)$, we obtain: 
  \begin{center}   $\lambda'_{j} \mu_{j} \sum_{l} \nu_{l}  \Vert  \rho_{l} \Vert^{2} \ge 0 $ \quad $\forall (\rho_{l}) \in \bigoplus H_{l}$ \end{center}
Choosing all but one $\rho_{l}$ equal to zero, we have $\lambda'_{j}  \nu_{l} \mu_{j} > 0 $,  and so $ \nu_{l} \mu_{j} > 0 $. 
$(6)$ Let $Z = (z_{rt})$, with $z_{rt} = b^{\star}_{ji} \pi_{j} (a^{\star}_{ \alpha 0} a_{\alpha 0}) \pi_{j}(x_{r}^{\star }x_{t}) b_{ji}$, and $x_{r} \in  \M(I^{c})$. \\Ê$Z$ is a positive operator-valued matrix, so by the same process, induction and intertwining, we get: 
\begin{center} $\sum \lambda_{k} \nu_{l} \mu_{s}  (a_{ls}\eta_{s} , a_{lj}\eta_{j} )   =  (\pi_{j} (a^{\star}_{ \alpha 0} a_{\alpha 0}) \eta_{j} , \eta_{j} )$       \end{center}
Since it's true for all $\eta_{s} \in \bigoplus H_{s}$, all the term with $s \ne j$ are null: 
\begin{center} $ (\pi_{j} (a^{\star}_{ \alpha 0} a_{\alpha 0}) \eta_{j} , \eta_{j} )= \sum \lambda_{k} \nu_{l} \mu_{j}  (a_{lj}\eta_{j} , a_{lj}\eta_{j} ) $       \end{center}
But, we know that $\nu_{l} \mu_{j}  > 0$, then, by induction hypothesis; 
\begin{center} $\pi_{j} (\bar{a}_{0  \alpha} a_{\alpha 0})  = \sum \Lambda_{l}  \bar{a}_{jl} a_{lj}$,   with    $\Lambda_{l} > 0 $   \end{center}
The result follows because $\alpha$ is a weak generator and $j \in \langle  \alpha , i \rangle $.
  \end{proof}
  \begin{corollary} (Connes fusion for charge $\alpha$)  
   \begin{displaymath} H_{\alpha}  \boxtimes H_{i}  =  \bigoplus_{j \in \langle  \alpha , i \rangle} H_{j} \end{displaymath} \end{corollary}
  \begin{proof}  Let $\X_{0} \subset Hom_{\M(I^{c})}(H_{0} , H_{\alpha})$, be the linear span of intertwiners $x = \pi_{\alpha}(h) a_{\alpha 0}$, with $h \in \M(I)$ and $a_{\alpha 0}$  a primary field localised in $I$. Since $x \Omega = (\pi_{\alpha}(h) a_{\alpha 0}\pi_{0}(h)^{\star})\pi_{0}(h) \Omega$ with $h$ unitary, and  $\pi_{\alpha}(h) a_{\alpha 0}\pi_{0}(h)^{\star}$ also a primary field, it follows by the Reeh-Schlieder theorem (and by the fact that the unitary operators generate the von Neumann algebra) that $\X_{0}\Omega$ is dense in $\X_{0}H_{0}$. Now, using the von Neumann density in the same way that for the lemma \ref{density45},  $\X_{0}\Omega$ is also dense in $H_{\alpha}$. Let $x = \sum \pi_{\alpha}(h^{(r)})a_{\alpha 0}  \in \X_{0}$,  $x_{ji} = \sum \pi_{j}(h^{(r)})a_{ji}^{(r)}$ and $y \in \Y := Hom_{\M(I)}(H_{0} , H_{i})$. By the transport formula:
 $(x^{\star}xy^{\star}y \Omega , \Omega ) = (y^{\star} \pi_{i}(x^{\star}x)y\Omega , \Omega) = \sum \lambda_{j} \Vert x_{ji}y\Omega \Vert^{2}$.  
  Now, polarising this identity, we get an isometry $U$ of the closure of $\X_{0} \otimes \Y$ in $H_{\alpha} \boxtimes H_{i}$ into $\bigoplus H_{j}$, sending $x\otimes y$ to $\bigoplus \lambda^{1/2}_{j} x_{ji}y\Omega$. By the Hilbert space continuity lemma, $\X_{0} \otimes \Y$ is dense in $H_{\alpha}  \boxtimes H_{i}$. Now, each $a_{ji}$ can be non-zero, so by the unicity of the decomposition into irreducible, $U$ is surjective and then a unitary operator.   \end{proof} 
  \begin{corollary} (Commutativity for charge $\alpha$)  
  \begin{center} $H_{\alpha}  \boxtimes H_{i}  =  H_{i} \boxtimes H_{\alpha}  $ \end{center} \end{corollary}
  \begin{proof} We prove in the same way that $H_{i} \boxtimes H_{\alpha}  =  \bigoplus_{j \in \langle  \alpha , i \rangle} H_{j}$.   \end{proof}
  \subsection{Connes fusion with $H_{\beta}$ }  \label{nearweak}
 Recall that $\beta = (0,1)$ and  $\phi_{\alpha, \beta}^{\alpha} $ is non-zero.
  \begin{center}  $ \phi_{ij}^{\alpha}(z) \phi_{jk}^{\beta}(w)   = \sum \lambda_{l}  \phi_{il}^{\beta}(w) \phi_{lk}^{\alpha}(z)$   with  $\lambda_{l}  \ne 0$ iff  $l \in \langle  \beta , i \rangle \cap \langle \alpha , k \rangle$    \end{center}
  \begin{remark} We proceed  as previously: this braiding pass to the local primary field, we make principal and non-principal part  incorporating the phase correction. Now, $\beta$ is not  a weak generator, so, to prove a transport formula, we prove by induction on $\vert i \vert$ that $ a_{i0} c_{\beta 0 }^{\star} c_{\beta 0 } = [\sum \lambda_{l} c_{li }^{\star} c_{l i }]a_{i0}$, with $a_{i0}$ a chain of even primary field of charge $\alpha$ localised on $I$, $(c_{ij})$ even primary fields of charge $\beta$ localised on $I^{c}$, and    $\lambda_{l} \ge 0$ iff $l \in  \langle  \beta , i \rangle$. The proof uses the same arguments with positive operators... then by intertwining we obtain the following partial transport formula, and next, a partial  fusion rules:  \end{remark}
 \begin{corollary}  (Transport formula)  \begin{displaymath} \pi_{i} (\bar{c}_{0 \beta}.a_{\alpha 0}) = \sum_{j \in \langle  \beta , i \rangle} \lambda_{j} \bar{c}_{ij}.c_{ji} \quad   \textrm{with} \  \lambda_{j} \ge 0.  \end{displaymath} \end{corollary}
 \begin{corollary} (partial Connes fusion for $\beta$) 
  \begin{displaymath} H_{\beta}  \boxtimes H_{i}  \le  \bigoplus_{j \in \langle  \beta , i \rangle} H_{j} \end{displaymath}  \end{corollary}  
  \subsection{The fusion ring} 
    We define the fusion ring $(\T_{m} , \oplus , \boxtimes)$ generated  as the $\ZZZ$-module, by the discrete series of $\Vir_{1/2}$ at fixed charge $c_{m}$, with $m = \ell + 2$
  \begin{lemma}(closure under fusion)
  \begin{enumerate}  
  \item[(a)]  Each $H_{i}$ is contains in some $H_{\alpha}^{\boxtimes n}$. 
  \item[(b)]  The $H_{i}$'s are closed under Connes fusion.
   \item[(c)]  $H_{i} \boxtimes H_{j} = \bigoplus m_{ij}^{k} H_{k}$ with $m_{ij}^{k} \in \NNN$
  \end{enumerate}    \end{lemma}
  \begin{proof} (a)   Direct because $\alpha$ is a weak generator.\\ 
  (b)   Since $H_{i} \subset H_{\alpha}^{\boxtimes m}$  and  $H_{j} \subset H_{\alpha}^{\boxtimes n}$ for some $m, n$, we have $H_{i} \boxtimes H_{j} \subset H_{\alpha}^{\boxtimes m+n}$, which is, by induction, a direct sum of some $H_{i}$. Now, by Schur's lemma any subrepresentations of a direct sum of irreducibles, is a direct sum of irreducibles; then, so is for  $H_{i} \boxtimes H_{j} $. \\ 
  (c)  By induction,  $H_{\alpha}^{\boxtimes m+n}$ admits only finite multiplicities. \end{proof}
     \begin{definition} (Quantum dimension) A quantum dimension is an application $d : \T_{m} \to \RRR \cup \{ \infty\}$, which is additive and multiplicative for $\oplus$ and $\boxtimes$, and positive (possibly infinite) on the base $(H_{i})$.   \end{definition}
     \begin{reminder} On a fusion ring, finite as $\ZZZ$-module, the quantum dimension $d$ is finite if $\forall A \in \T_{m}$, $\exists B \in \T_{m}$ such that $H_{0} \le A \boxtimes B $. If so, $B$ is unique and called the dual of $A$, noted $A^{\star}$.   \end{reminder}
     \begin{remark} $H_{0}  \le H_{\alpha} \boxtimes H_{\alpha}$. Then,  $H_{\alpha}^{\ell} $ is self-dual and $d(H_{\alpha})$ finite. \end{remark}
          \begin{corollary} The quantum dimension is finite on the fusion ring.    \end{corollary}
     \begin{proof}  Because $H_{\alpha}$ is a weak generator, $\forall i$, $H_{i} \le H_{\alpha}^{\boxtimes n}$ for some $n$, then $d(H_{i}) \le d(H_{\alpha})^{n}$ finite.   \end{proof}
       \begin{reminder} (Frobenius reciprocity) If $nA \le B \boxtimes C$ then $nC \le B^{\star} \boxtimes A$.   \end{reminder}
   \begin{reminder} (Perron-Frobenius theorem) An irreducible matrix with positive entries admits one and only one positive eigenvalues. The corresponding eigenspace is generated by a single vector $v = (v_{i})$, with $v_{i} > 0$.     \end{reminder}
       \begin{corollary} A quantum dimension on $\T_{m}$ with $d(H_{0}) = 1$ is uniquely determined, and given by the fusion matrix of $H_{\alpha} = H^{\star}_{\alpha}$.  \end{corollary}
     \begin{proof} $H_{\alpha} \boxtimes (\sum d(H_{j})H_{j}) = \sum n_{\alpha j}^{k} d(H_{j}) H_{k} = \sum d(\sum n_{\alpha j}^{k} H_{j}) H_{k} \\Ê=  \sum d(\sum n_{\alpha k}^{j} H_{j}) H_{k} =  \sum d(H_{\alpha} H_{k}) H_{k} = d(H_{\alpha}) (\sum d(H_{k})H_{k})$.  \\Note that $n_{\alpha j }^{k} = n_{\alpha k}^{j}$ is  immediate from Frobenius reciprocity and $H_{\alpha}$ self-dual.
     Next, $\alpha$ is a weak generator, so the fusion matrix $M_{\alpha}$, is irreducible. The result follows with  the Perron-Frobenius theorem,
    with $v_{i} = d(H_{i})$. \end{proof}
      \subsection{The fusion ring and index of subfactor.} 
    \begin{definition} Let $\langle a , b \rangle_{n} = \{  c = \vert a-b \vert, \vert a-b \vert + 1, ..., a+b  \ \vert \ a+b+c \le n  \}$. \end{definition} 
   \begin{corollary} (Connes fusion rules for $\alpha$ and $\beta$)   \label{immediat}
\begin{displaymath} 
 (a) \quad H_{\alpha}^{\ell}   \boxtimes   H_{i'j'}^{\ell} = \bigoplus_{ (i'', \hspace{0,1cm}  j'') \in \langle \1/2 , i'  \rangle_{\ell}  \times \langle \1/2 , j'  \rangle_{\ell + 2}  } H_{i''j''}^{\ell}   
 \end{displaymath}   
 \begin{displaymath} 
(b) \quad H_{\beta}^{\ell}   \boxtimes   H_{i'j'}^{\ell} \le  \bigoplus_{ (i'', \hspace{0,1cm}  j'') \in \langle 0 , i'  \rangle_{\ell}  \times \langle 1 , j'  \rangle_{\ell + 2}  } H_{i''j''}^{\ell}   
 \end{displaymath}   

  \end{corollary}
  \begin{proof}  Immediate from theorem \ref{woker} and sections \ref{weakalpha}, \ref{nearweak}. \end{proof}
  \begin{reminder} (Connes fusion rules for $L\gg$ at level $\ell$ \cite{2})
\begin{displaymath}   H_{i}^{\ell}   \boxtimes   H_{j}^{\ell} = \bigoplus_{ k \in \langle i , j  \rangle_{\ell}   } H_{k}^{\ell}          \end{displaymath}   
  \end{reminder}
  \begin{reminder}(Quantum dimension \cite{2})
\begin{center}  $d( H_{i}^{\ell}  )  = \frac{sin(p \pi /  m)}{sin(\pi / m)}$  \end{center}
  with $m = \ell + 2$ and $p = dim(V_{i}) = 2i+1$.
     \end{reminder}
  \begin{definition} Let $(\R_{\ell} , \oplus , \boxtimes)$ be the fusion ring generated as $\ZZZ$-module by  discrete series of $LSU(2)$ at level $\ell$.  \end{definition}
  \begin{remark}  $H_{pq}^{m}$ and $H_{m-p , m+2-q}^{m}$ are the same representation of $\Vir_{1/2}$ because $h_{pq}^{m}$ and $h_{m-p , m+2-q}^{m}$.   \end{remark}
\begin{definition}  Let  $\tilde{\T}_{m} $ be a formal associative fusion ring, generated by $(\tilde{H}_{pq}^{m})$ (or $(\tilde{H}_{ij}^{\ell})$ with the other notation), with every $\tilde{H}_{pq}^{m} $ distinct (in particular  $\tilde{H}_{pq}^{m} \ne \tilde{H}_{m-p, m+2-q}^{m}$), using the fusion rules of corollary \ref{immediat}.   \end{definition}
\begin{proposition} The ring  $\tilde{\T}_{m} $ is isomorphic to $\R_{\ell} \otimes_{\ZZZ} \R_{\ell +2}$.    \end{proposition}
\begin{proof} Let the bijection $\varphi: \tilde{\T}_{m} \to \R_{\ell} \otimes_{\ZZZ} \R_{\ell +2}$ with $\varphi (\tilde{H}_{ij}^{\ell} ) = ( H_{i}^{\ell}  , H_{j}^{\ell})$.
The fusion matrix of  $\tilde{H}_{\alpha}^{\ell} $ is clearly equal to the fusion matrix of $(H_{1/2}^{\ell} , H_{1/2}^{\ell +2} )$. Then, by Perron-Frobenius theorem, 
$\tilde{H}_{ij}^{\ell} $ and $(H_{i}^{\ell}  , H_{j}^{\ell} )$ has the same quantum dimension. Now, $d(\tilde{H}_{\beta}^{\ell}  ).d(\tilde{H}_{i'j'}^{\ell} ) \le \sum d(\tilde{H}_{i''j''}^{\ell})$, and $d(H_{0}^{\ell}  , H_{1}^{\ell} ).d(H_{i'}^{\ell}  , H_{j'}^{\ell} ) = \sum d(H_{i''}^{\ell}  , H_{j''}^{\ell} ) $. So, by positivity, the previous inequality is an equality and: 
\begin{center}  $\tilde{H}_{\beta}^{\ell}   \boxtimes   \tilde{H}_{i'j'}^{\ell} =  \bigoplus_{ (i'', \hspace{0,1cm}  j'') \in \langle 0 , i'  \rangle_{\ell}  \times \langle 1 , j'  \rangle_{\ell + 2}  } \tilde{H}_{i''j''}^{\ell}   $  \end{center}
So, the fusion rules for $\tilde{H}_{\beta}^{\ell}$ is also the same that for $(H_{0}^{\ell}  , H_{1}^{\ell} )$.   Now, by associativity, the fusion rules for $\tilde{H}_{\alpha}^{\ell}$ and $\tilde{H}_{\beta}^{\ell}$ give all the fusion rules. \\ The result follows.
  \end{proof}
  \begin{corollary}   $\T_{m}$  is isomorphic to the subring of  $(\R_{\ell} \otimes_{\ZZZ} \R_{\ell +2})\otimes \QQQ $  generated by $\1/2[(H_{i}^{\ell}  , H_{j}^{\ell} ) + (H_{\frac{\ell}{2} - i}^{\ell}  , H_{\frac{\ell +2}{2} - j}^{\ell} ) ]$; or to  $(\R_{\ell} \otimes_{\ZZZ} \R_{\ell +2}) / (  (H_{i}^{\ell}  , H_{j}^{\ell} )- (H_{\frac{\ell}{2} - i}^{\ell}  , H_{\frac{\ell +2}{2} - j}^{\ell} ) )$. In particular, the fusion is commutative. \end{corollary}
  \begin{proof}Immediate.  \end{proof}
    
\begin{theorem} (Connes fusion for $\Vir_{1/2}$) 
\begin{displaymath}   H_{ij}^{\ell}   \boxtimes   H_{i'j'}^{\ell} = \bigoplus_{ (i'', \hspace{0,1cm}  j'') \in \langle i , i'  \rangle_{\ell}  \times \langle j , j'  \rangle_{\ell + 2}  } H_{i''j''}^{\ell}          \end{displaymath}   
  \end{theorem}
\begin{proof}Immediate.   \end{proof}
\begin{remark}  $H_{00}^{\ell} \le (H_{ij}^{\ell}) ^{\boxtimes 2}$, so that  $H_{ij}^{\ell}$ is self-dual.    \end{remark}
\begin{theorem} (Quantum dimension for $\Vir_{1/2}$) 
\begin{center}  $d( H_{ij}^{\ell}  )  =  d(H_{i}^{\ell}).d(H_{j}^{\ell +2} )= \frac{sin(p \pi /  m)}{sin(\pi / m)} . \frac{sin(q \pi /  (m+2))}{sin(\pi / (m+2))}$  \end{center}
with $m = \ell + 2$, $p =  2i+1$ and $q =  2j+1$.  \end{theorem}
\begin{proof}Immediate.   \end{proof}
\begin{theorem} (Jones-Wassermann subfactor) 
\begin{displaymath} \pi_{ij}^{\ell}(\Vir_{1/2}(I))'' \subset  \pi_{ij}^{\ell}(\Vir_{1/2}(I^{c}))^{\natural}       \end{displaymath}    
It's a finite depth, irreducible,  hyperfinite III$_{1}$-subfactor, isomorphic to  the hyperfinite III$_{1}$-factor $\R_{\infty}$ tensor  the II$_{1}$-subfactor : 
\begin{displaymath}  (\bigcup \CCC \otimes End_{\Vir_{1/2}} (H_{ij}^{\ell})^{\boxtimes n})'' \subset (\bigcup End_{\Vir_{1/2}} (H_{ij}^{\ell})^{\boxtimes n+1})''  \  \textrm{of index $d( H_{ij}^{\ell} )^{2}$.}     \end{displaymath}  
 \end{theorem}
 \begin{proof} It's finite depth because there is only finitely many irreducible positive energy representations of charge $c_{m}$. Next, the hyperfinite III$_{1}$-subfactor and the irreducibility has already been proven before. The higher relative commutants can be calculated using the method of H. Wenzl \cite{wenzl}. The rest follows from the work of S. Popa \cite{popa}.    \end{proof}

  \end{document}